\documentclass[11pt,a4paper]{article}

\usepackage{amsmath}
\usepackage{amsthm}
\usepackage{amsfonts}
\usepackage{paralist}

\usepackage{hyperref}
\hypersetup{%
plainpages=false,
colorlinks,urlcolor=blue,linkcolor=blue,citecolor=blue,
pdftitle={Stochastic sequences with a renegerative structure that may
  depend both on the future and on the past},
pdfauthor={Sergey Foss and Stan Zachary},
pdfsubject={},
pdfpagemode={UseOutlines},
pdfpagelayout={OneColumn},
pdfstartview={FitBH}
}

\parskip\smallskipamount
\newtheorem{theorem}{Theorem}
\newtheorem{lemma}[theorem]{Lemma}
\newtheorem{corollary}[theorem]{Corollary}
\newtheorem{proposition}[theorem]{Proposition}
\theoremstyle{remark}
\newtheorem{remark}{Remark}
\newtheorem{example}{Example}

\renewcommand{\Pr}{\mathbf{P}}
\newcommand{\E}{\mathbf{E}}

\newcommand{\Z}{\mathbb{Z}}
\newcommand{\N}{\mathbb{N}}
\newcommand{\I}{\textbf{I}}

\newcommand{\si}{\mathcal{\sigma}}

\parskip \smallskipamount

\title{Stochastic sequences with a regenerative structure that may depend both
  on the future and on the past}
\author{Sergey Foss$^{1}$ and Stan Zachary$^{1}$}
\date{\today}

\begin{document}

\maketitle
\stepcounter{footnote}\footnotetext{School of Mathematics and Computer
  Sciences and the Maxwell Institute for Mathematical Sciences,
  Heriot-Watt University, Edinburgh, EH14 4AS, Scotland, UK.  E-mail:
  s.foss@hw.ac.uk and s.zachary@gmail.com.  Research of both authors was
  partially supported by EPSRC grant EP/I017054/1.}

\begin{quotation}\small
  Many regenerative arguments in stochastic processes use random times
  which are akin to stopping times, but which are determined by the
  future as well as the past behaviour of the process of interest.
  Such arguments based on ``conditioning on the future'' are usually
  developed in an ad-hoc way in the context of the application under
  consideration, thereby obscuring underlying structure.  In this
  paper we give a simple, unified and more general treatment of
  such conditioning theory.  We further give a number of novel applications
  to various particle system models, in particular to various flavours
  of contact processes and to infinite-bin models.  We give a number
  of new results for existing and new models.  We further make
  connections with the theory of Harris ergodicity.
\end{quotation}



\section{Introduction}
\label{sec:introduction}

Many arguments in stochastic processes use random times akin to
stopping times to establish regenerative or ergodic behaviour.  These
may be randomised stopping times as in the theory of Harris
ergodicity.  Alternatively they may be random times in which there is
an element of probabilistic conditioning on the, possibly infinite,
future of the
process of interest, but in which this conditioning is sufficiently
controlled that, with respect to these random times, the process
behaves as if they were stopping times; such times are used, for
example, in establishing the long-term behaviour of particle systems
and population processes conditional on their survival, and in
establishing the behaviour of processes conditioned to avoid given
regions of their state spaces.  More generally such random times,
defined by conditioning on future behaviour may also be used to
establish the unconditional behaviour of their parent processes---as
we illustrate in the applications of
Sections~\ref{sec:right-endp-supercr} and
\ref{sec:infinite-bin-model-1}.  Such arguments based on
``conditioning on the future'' are usually developed from scratch and
in an ad-hoc way in the context of the application under
consideration, thereby to some extent obscuring underlying structure.

Our aim in the present paper is to give a unified treatment of these
phenomena.  In doing so we develop a simple theory which is more
general than the sum of those already existing, and which has
applications---for example to some variants of the models considered
in Sections~\ref{sec:right-endp-supercr} and
\ref{sec:infinite-bin-model-1}---which cannot be managed by the simple
application of such bits of theory as already exist.  We further apply
the results obtained to a number of new models, including variants of
the three-state contact process of
Section~\ref{sec:right-endp-supercr} and of the infinite-bin models of
Section~\ref{sec:infinite-bin-model-1}.

For simplicity of exposition we work in this paper in discrete time.
In general the process of interest $\{X_n\}_{n\ge0}$, say, may be
\emph{constructed} as a functional of an underlying process
$\{\xi_n\}_{n\ge1}$, or $\{\xi_n\}_{n\in\Z}$, where as usual $\Z$ is
the set of integers.  In the present paper we assume that the process
$\{\xi_n\}$ consists of independent identically distributed random
variables~$\xi_n$.
However, some of the phenomena studied here for the
process $\{X_n\}$ continue to occur under more relaxed assumptions for
the underlying process $\{\xi_n\}$, for example, that it 
is regenerative in the sense that there are some random times at which
the process starts anew independently of the past.  These extensions
are typically straightforward; for example, in the case where the
process $\{\xi_n\}$ is regenerative we may restrict arguments to the
regeneration times. Extensions to continuous time are also
straightforward provided that the process $\{\xi_n\}$ is replaced by
something satisfying analogous homogeneity and independence
conditions; in the case of interacting particle systems (see
Section~\ref{sec:right-endp-supercr}) this is typically the collection
of underlying Poisson processes.

In Section~\ref{sec:cond-infin-future} we present our basic
theory. Our aim is to identify sequences of random times
$0\le\tau_0<\tau_1<\dots$, the definition of each of which may depend
both on the (typically infinite) past and on the (typically infinite)
future, but which are nevertheless such that the segments of the
process $\{\xi_n\}$ between successive such times are i.i.d..
Following Kuczek~\cite{Kuc}, we shall refer to these times (which are
an instance of regeneration times) as
\emph{break times}.

It is helpful to give an immediate example (in which the dependence is
on the future only).
\begin{example}\label{ex:1}
  Let $\{\xi_n\}_{n\ge1}$ consist of i.i.d.\ random variables, with
  common distribution given by
  \begin{align*}
    \Pr(\xi_i = 1) & = p\\
    \Pr(\xi_i = -1) & = q\\
    \Pr(\xi_i = 0) & = 1 - p - q,
  \end{align*}
  where $0<q<p$ and $p+q<1$.  We consider three variant constructions
  of random times whose definitions involve conditioning on the
  (infinite) future.
  \begin{compactenum}[(a)]
  \item For each $n\ge0$ let $F_n$ be the ``future'' event that
    $\sum_{i=1}^m \xi_{n+i}\ge0$ for all $m\ge1$.  Note that the
    common probability of the events $F_n$ is strictly positive.  Let
    $0\le\tau_0<\tau_1<\dots$ be the successive times $n$ at which the
    event $F_n$ occurs.  Then it is easy to see (and is a special case
    of the result of Example~\ref{ex:rw2} below) that the successive
    segments $\{\xi_{\tau_k+1},\dots,\xi_{\tau_{k+1}}\}$, $k\ge0$, of
    the process $\{\xi_n\}$ are independent and identically
    distributed in $k$.  In particular the successive time intervals
    $\tau_{k+1}-\tau_k$, $k\ge0$, are i.i.d.
  \item Now suppose that, for each $n\ge0$, we let $F'_n$ be the
    future event that $\sum_{i=1}^m \xi_{n+i}\ge0$ for all $m\ge1$ and
    additionally $\xi_{n+2}=0$.  Again the common probability of the
    events $F'_n$ is strictly positive.  Let
    $0\le\tau'_0<\tau'_1<\dots$ be the successive times $n$ at which
    the event $F'_n$ occurs.  In this case we do \emph{not} have
    independence of the successive segments
    $\{\xi_{\tau'_k+1},\dots,\xi_{\tau'_{k+1}}\}$, $k\ge0$; for
    example, if the events $F'_0$ and $F'_1$ both occur, then
    necessarily $\xi_{\tau'_1+1}=\xi_2=0$.
  \item Finally, suppose that the events $F'_n$ are as in (b)
    above. However, define the sequence $0\le\tau''_0<\tau''_1<\dots$
    by $\tau''_0=\min\{n\ge0: F'_n \text{ occurs}\}$ (i.e.\
    $\tau''_0=\tau'_0$) and, for $k\ge1$,
    $\tau''_k=\min\{n\ge\tau''_{k-1}+2: F'_n \text{ occurs}\}$.  Then
    it is again easy to see that the successive segments
    $\{\xi_{\tau''_k+1},\dots,\xi_{\tau''_{k+1}}\}$, $k\ge0$, of the
    process $\{\xi_n\}$ are once more independent and identically
    distributed in $k$.
  \end{compactenum}
\end{example}

The reasons for the different behaviours in the above example is that,
in order to obtain i.i.d.\ behaviour, we require the definitions of
the successive time $\tau_k$ to satisfy a form of monotonicity
condition in which, in a sense which we make clear in
Section~\ref{sec:cond-infin-future}, information about the future does
not cumulate; this condition is satisfied in the variants (a) and (c)
of Example~\ref{ex:1}, but not in the variant~(b).  In
Section~\ref{sec:cond-infin-future} we develop the relevant theory in
a general setting in which the break times $\tau_k$ may
depend on both the past and the future behaviour of the underlying
process $\{\xi_n\}$.  In particular we give conditions for the
segments of the process~$\{\xi_n\}$ between break times to
constitute i.i.d.\ cycles.  We believe this theory to be novel in the
general setting.  As a simple example, we apply the theory to a
general random walk with positive drift (generalising
Example~\ref{ex:1}).

In Section~\ref{sec:right-endp-supercr} we give applications of our
theory to a number of discrete-time contact process models, both
showing how existing results are more readily understood, and giving
some new results for a three-state contact process.  The theory is
equally applicable in the continuous-time setting, and has
applications in general to particle systems and similar models in
which processes ``survive'' with probabilities strictly between zero
and one.  In Section~\ref{sec:infinite-bin-model-1} we give
applications of the theory of Section~\ref{sec:cond-infin-future} to a
class of ``infinite-bin'' models.

In Section \ref{sec:HR} we make some connections with the
existing theory of Harris-ergodic Markov chains. Finally, in Section 
\ref{sec:discussion} we discuss a number of
other models and extended applications, including conditioning,
scaling and regeneration/asymptotic stationarity of the driving
sequence $\{\xi_n\}$.

\section{Conditioning on the future}
\label{sec:cond-infin-future}

We assume that the underlying process $\{\xi_n\}_{n\in\Z}$ (defined on
some underlying probability space $(\Omega,\,\mathcal{F},\,\Pr)$)
consists of i.i.d.\ random variables $\xi_n$.
For two events $A$ and $B$, we write $A=B$ if
their symmetric difference, $A\triangle B = A\setminus B \cup B\setminus A$,
has probability zero.

For $m\le n$ denote by ${\si}_{m,n}$ the $\sigma$-algebra generated
by $\xi_m,\dots,\xi_n$, and let ${\si}_n={\si}_{-\infty,n}$.  The
process $\{X_n\}_{n\in Z}$ (or $\{X_n\}_{n\in\Z_+}$) of interest will
typically be defined in terms of the process $\{\xi_n\}_{n\in\Z}$ and
adapted with respect to the filtration $\{{\si}_n\}_{n\in Z}$; for
example it may be defined by the stochastic recursion
\begin{equation}\label{SRS}
  X_{n+1}=f(X_{n},\xi_{n+1})
\end{equation}
for some function $f$ (and hence homogeneous Markov)

Define also $\sigma$ to be the sigma-algebra generated by all the
random variables $\xi_n$, $-\infty < n < \infty$.  As usual, we may
introduce a measure-preserving shift transformation $\theta$ on
$\sigma$-measurable random variables by assuming that $
\xi_n\circ\theta=\xi_{n+1}$, for all $n$, and that, more generally,
$g(\xi_m,\dots,\xi_n) \circ \theta = g(\xi_{m+1},\dots,\xi_{n+1})$,
for any measurable function $g$. (Here the finite sequence of random
variables $\{\xi_m,\dots,\xi_n\}$ may also be replaced a half-infinite
or an infinite one.)

We may further extend the shift transformation to events in $\sigma$
by defining (with a slight abuse of notation) $G_1 \circ \theta =G_2$
if $\I_{G_1}\circ \theta = \I_{G_2}$.  (Here by $\I_{G_1}$ we denote
the indicator function of the event $G_1$ which equals $1$ if the
event $G_1$ occurs and $0$ otherwise).  We then say that a sequence
of events $\{G_n\}$ is \emph{stationary} if it is so with respect to
$\theta$, i.e. $G_n\circ\theta=G_{n+1}$, for all $n$. Therefore, 
if there are two sequences and each of them is stationary, then also
they are jointly stationary.

In what follows, we consider,  
in addition to the process $\{\xi_n\}$, a given
sequence of events $\{F_n\}_{n\in\Z_+}$ which always satisfies the
following conditions:

\begin{compactenum}[(F1)]
\item {\it The sequence $\{F_n\}_{n\in\Z_+}$ is stationary, with the
    common value of $\Pr(F_n)$ strictly positive.}
\item {\it For each $n$, the event $F_n$ is defined in terms of the
    ``future'' process $\{\xi_m\}_{m>n}$, i.e.\
    $\I_{F_n}=g(\xi_{n+1},\xi_{n+2},\dots)$ for some function $g$
    (which by stationarity is independent of~$n$). Here the future
    dependence of each of the events $F_n$ may be on either the finite
    or the infinite future.}
\end{compactenum}

It is our intention to define, in terms of the process $\{\xi_n\}$ and
the sequence $\{F_n\}$, a sequence of random times
$0\le\tau_0<\tau_1<\dots$ on the nonnegative integers $\Z_+$.  Our
interest is in the behaviour of the successive segments of processes
$\{\xi_{\tau_k+1},\dots,\xi_{\tau_{k+1}}\}$, $k\ge0$.  (Note that the
definition of any such segment
$\{\xi_{\tau_k+1},\dots,\xi_{\tau_{k+1}}\}$ includes a specification
of its length $\tau_{k+1}-\tau_k$.)  It is convenient to take a
``point process'' approach, and to define first a further sequence of
events $\{A_n\}_{n\in\Z_+}$; the times $\tau_k$ are then defined to be
the successive times $n\ge0$ such that the event $A_n$ occurs.
(We observe that the need to unambiguously index the times $\tau_k$
obliges us to choose some origin~$0$ of time, and it is then
convenient to restrict attention to behaviour subsequent to time~$0$.
However, to the extent that the times of occurrence of the events
$A_n$ may be viewed as a point process on the positive integers, much
of what follows may be extended without difficulty to the entire set
$\Z$ of all the integers.)

\begin{remark}\label{rmk:mrf}
  Insofar as the times of occurrence of the events $A_n$ may be
  regarded as a point process on the integers, the determination of
  their locations, in the results below, is the result of
  simultaneous, and essentially Markovian, conditioning from both the
  past and the future.  There are therefore connections with the
  theory of one-dimensional Markov random fields.  However, for our
  present purposes, it is natural to define the conditions for these
  results directly in terms of events.  The distribution of the point
  process is then induced by the underlying i.i.d.\ driving sequence
  $\{\xi_n\}$.
\end{remark}

Our main result of this section is now the following theorem. For the
result to hold, we require some conditions which imply, in particular,
that each event $A_n$ may be represented as an intersection
$A_n=H_n\cap F_n$ of a ``past'' event $H_n\in\si_n$ and the ``future''
event $F_n$ defined earlier (see Remark~\ref{rmk3} below).

\begin{theorem}\label{th:1}
  Let a sequence of events $\{A_n\}_{n\in\Z_+}$ be given.  Let
  $\tau_0=\min \{n\ge 0: {\bf I}_{A_n}=1 \}$ and, for $k\ge 0$, let
  $\tau_{k+1} = \min \{n>\tau_k: {\bf I}_{A_n}=1 \}$.  Assume that
  $\tau_k<\infty$ a.s., for all $k$.

  Let also the following be given: a sequence of ``future'' events
  $\{F_n\}_{n\in\Z_+}$ satisfying the earlier conditions (F1) and (F2),
  sequences $\{H'_n\}_{n\in\Z_+}$ and $\{H''_n\}_{n\in\Z_+}$ of
  ``past'' events such
  that, for each $n$, we have $H'_n\in\si_n$ and $H''_n\in\si_n$, and,
  finally, an array of events $\{E_{n,n+m}\}_{n\ge0,m>0}$ with
  each $E_{n,n+m}\in\si_{n+1,n+m}$ and such that, for each fixed $m$,
  the sequence $\{E_{n,n+m}\}_{n\ge0}$ is stationary.  Suppose further
  that all these sequences are linked by the following relations: for
  $n\ge 0$,
  \begin{equation}
    \label{eq:35}
   \{\tau_0=n\} \equiv A_0^c \cap \dots \cap A_{n-1}^c \cap A_n = H'_n \cap F_n,
  \end{equation}
  and, for $0\le n' < n$, 
  \begin{equation}
    \label{eq:36}
    \{\exists \ k : \tau_k=n', \tau_{k+1}=n\} \equiv 
    A_{n'} \cap A_{n'+1}^c \cap \dots \cap A_{n-1}^c \cap A_{n}
    = H''_{n'} \cap E_{n',n} \cap F_{n}.
  \end{equation}
  Then the successive segments of processes
  $\{\xi_{\tau_k+1},\dots,\xi_{\tau_{k+1}}\}$, $k\ge0$, are
  independent and identically distributed.   
  In particular the random variables $\tau_{k+1}-\tau_k$, $k\ge0$, are
  i.i.d..  Further, for some constant $a>0$, and for all $k\ge0$,
  \begin{equation}
    \label{eq:37}
    \Pr(\tau_{k+1}-\tau_k = n) = a\Pr(E_{0,n}).
  \end{equation}
\end{theorem}

Note 
that it follows directly from
\eqref{eq:35} and \eqref{eq:36} that, for all $n$, we have
$A_n\subseteq F_n$, and from \eqref{eq:36} that, for all $n'$, we have
$A_{n'}\subseteq H''_{n'}$.  Hence we note
\begin{equation}
  \label{eq:38}
  A_n \subseteq H''_n \cap F_n, \qquad n\ge0.
\end{equation}

\begin{proof}[Proof of Theorem~\ref{th:1}]
  Fix $k\ge0$.  For any $0\le n_0<\dots<n_{k+1}$, it follows from
  \eqref{eq:35}, \eqref{eq:36} and the observation that $A_n\subseteq
  F_n$ for all $n$, that the following representation holds: 
  \begin{align}
    \{\tau_0=n_0,\dots,\tau_{k}=n_{k},\,\tau_{k+1}=n_{k+1}\}
    \hspace{-30ex} & \nonumber\\
    & = H'_{n_0} \cap H''_{n_0} \cap E_{n_0,n_1} \cap H''_{n_1} \cap
    \dots \cap H''_{n_k} \cap E_{n_k,n_{k+1}} \cap F_{n_{k+1}}
    \label{eq:HEF}
  \end{align}
  where the intersection of all but the last two events in \eqref{eq:HEF}, say 
  $\widetilde H_{n_k}$, belongs to the sigma-algebra $\si_{n_k}$ and is the
  same for all values of 
  $n_{k+1}$.  Thus for any events $G_{n_k}\in\si_{n_k}$,
  $G_{n_k,n_{k+1}}\in\si_{n_k+1,n_{k+1}}$,
  \begin{align}
    \Pr(\{\tau_0=n_0,\dots,\tau_{k}=n_{k},\,\tau_{k+1}=n_{k+1}\}%
    \cap G_{n_k}\cap G_{n_k,n_{k+1}})
    \hspace{-35ex} & \nonumber\\
    & = \Pr(G'_{n_k}\cap E_{n_k,n_{k+1}} \cap G_{n_k,n_{k+1}} \cap F_{n_{k+1}})
    \nonumber\\
    & = \Pr(G'_{n_k}) \Pr(E_{n_k,n_{k+1}} \cap G_{n_k,n_{k+1}}) \Pr(F_{n_{k+1}}).
    \label{eq:40}
  \end{align}
  where the event $G'_{n_k}=\widetilde H_{n_k}\cap G_{n_k}$ belongs to
  $\si_{n_k}$ and is the same for all values of $n_{k+1}$ and for all
  events $G_{n_k,n_{k+1}}$.  Thus, using also the stationarity of the
  sequence $\{F_n\}_{n\in\Z_+}$ and (for each $m$) of the sequence
  $\{E_{n,n+m}\}_{n\ge0}$, it follows that, conditional on
  $\{\tau_0=n_0,\dots,\tau_{k}=n_{k}\}$ and the process
  $\{\xi_n\}_{1\le n\le n_k}$, the distribution of
  $\{\xi_{\tau_k+1},\dots,\xi_{\tau_{k+1}}\}$ is that of
  $\{\xi_{\tau_0+1},\dots,\xi_{\tau_{1}}\}$ conditional on the
  occurrence of the event $\{\tau_0=n\}$, for any $n$ such that the
  latter probability is strictly positive.  All the assertions of the
  theorem now follow. In particular, to establish \eqref{eq:37}, we
  consider \eqref{eq:40} again with $G_{n_k}=G_{n_k,n_{k+1}}=\Omega$
  and with $n_{k+1}=n_k+n$ where $n$ is fixed.  Then we sum the
  expression \eqref{eq:40} all $0\le n_0<\ldots <n_k$ to obtain $\Pr
  (\tau_{k+1}-\tau_k=n)$.  This is of the form $a\Pr(E_{0,n})$ where,
  clearly, $a$ does not depend on $n$; therefore $a$ also does not
  depend on $k$ since the probabilities $a\Pr(E_{0,n})$ sum to
  one.
\end{proof}

\begin{remark}
  \label{rmk3}
  It is worth pausing to note, in somewhat intuitive terms, the
  significance of the conditions of Theorem~\ref{th:1}.  Note first
  that it follows straightforwardly from \eqref{eq:35} and
  \eqref{eq:36} that, for all $n$, we have 
  \begin{equation}\label{eq:AHF}
  A_n=H_n\cap F_n, \quad \mbox{with}
  \quad
  H_n = H'_n\cup\bigcup_{0\le n' <n} H_{n'}'' \cap E_{n',n} \in \si_n.
  \end{equation}  
  Suppose
  now that we proceed forwards in time with the process $\{\xi_n\}$.
  At each time $n$ such that the event $A_n$ occurs, we learn
  something about the future evolution of the process, namely that the
  event $F_n$ occurs.  In order to have some regeneration at this
  time, we need to ensure that, at each such time~$n$, given the knowledge
  that $F_n$ occurs, our knowledge at that time about the future
  probabilistic behaviour of the process is not \emph{further} conditioned by
  our knowledge of whether or not, for each earlier time $n'<n$, the
  event $A_{n'}$ occurs.
  This is essentially guaranteed by the condition~\eqref{eq:36}, which
  is in effect a form of monotonicity condition.
\end{remark}

We now give some special cases of Theorem~\ref{th:1}.
Corollary~\ref{cor:1} will be applied later to the two-state contact
process of Section~\ref{sec:two-state-contact} and to the basic
infinite-bin model of Section~\ref{sec:DIBModel}.

\begin{corollary}\label{cor:1}
  Suppose that, for all $n$, we have $A_n=F_n$ and that the sequence
  $\{F_n\}_{n\in\Z_+}$ satisfies the earlier conditions (F1) and (F2) and
  additionally the monotonicity condition
  \begin{equation}
    \label{eq:41}
    F_n \cap F_{n+m} = E'_{n,n+m} \cap F_{n+m}, \qquad n\ge 0,\ m>0,
  \end{equation}
  for some array of events $\{E'_{n,n+m}\}_{n\ge0,m>0}$ with each
  $E'_{n,n+m}\in\si_{n+1,n+m}$ and such that, for each fixed $m$, the
  sequence $\{E'_{n,n+m}\}_{n\ge0}$ is stationary.  Then the
  conclusions of Theorem~\ref{th:1} follow (with \eqref{eq:37} holding
  for appropriately defined $E_{0,n}$).
\end{corollary}

\begin{proof}
  Note first that \eqref{eq:41} implies that also
  \begin{equation}
    \label{eq:42}
    F_n^c \cap F_{n+m} = (E'_{n,n+m})^c \cap F_{n+m}, \qquad n\ge 0,\ m>0,
  \end{equation}
    Since $A_n=F_n$ for all $n$, it follows from \eqref{eq:41} and
  \eqref{eq:42} that, for all $n\ge0$,
  \begin{displaymath}
    A_0^c \cap \dots \cap A_{n-1}^c \cap A_n
    =  (E'_{0,n})^c \cap \dots \cap (E'_{n-1,n})^c \cap F_n
  \end{displaymath}
  and, for $0\le n' < n$,
  \begin{displaymath}
    A_{n'} \cap A_{n'+1}^c \cap \dots \cap A_{n-1}^c \cap A_{n}
    = E'_{n',n} \cap (E'_{n'+1,n})^c \cap \dots \cap (E'_{n-1,n})^c \cap F_{n}.
  \end{displaymath}
  Thus the conditions of Theorem~\ref{th:1} are readily seen to be
  satisfied with $\{F_n\}_{n\in\Z_+}$ as here, and for appropriately
  defined $\{H'_n\}_{n\in\Z_+}$, $\{H''_n\}_{n\in\Z_+}$ (with
  $H''_n=\Omega$ for all $n$) and $\{E_{n,n+m}\}_{n\ge0,m>0}$.
\end{proof}

\begin{remark}
  We can now provide some explanation for the observations of
  Example~\ref{ex:1} in the Introduction.  In the variant~(a), the
  condition~\eqref{eq:41} of Corollary~\ref{cor:1} is readily seen to
  be satisfied, with $E'_{n,n+m}=\{\sum_{i=1}^{m'}\xi_{n+i}\ge0\text{
    for all $1\le m'\le m$}\}$.  In the variant~(b) of
  Example~\ref{ex:1} the monotonicity condition~\eqref{eq:41} of
  Corollary~\ref{cor:1} (with the sequence $\{F_n\}_{n\ge0}$ replaced
  by $\{F'_n\}_{n\ge0}$) clearly cannot be satisfied, and more
  generally the conditions of Theorem~\ref{th:1} cannot be satisfied
  (for otherwise that theorem would contradict the known behaviour for
  this example).  However, for the variant~(c) of Example~\ref{ex:1}
  we may observe that the condition~\eqref{eq:41} of
  Corollary~\ref{cor:1} (again with the sequence $\{F_n\}_{n\ge0}$
  replaced by $\{F'_n\}_{n\ge0}$) is satisfied whenever $m\ge2$, with
  $E'_{n,n+m}=\{\sum_{i=1}^{m'}\xi_{n+i}\ge0 \text{ for all $1\le
    m'\le m$ and $\xi_{n+2}=0$}\}$.  Since the enforced minimum
  separation $\tau_k-\tau_{k-1}\ge2$ for $k\ge1$ implies that $A_n\cap
  A_{n+1}=\emptyset$ for all $n$, this restricted version of the
  condition~\eqref{eq:41}, coupled with the proof of
  Corollary~\ref{cor:1}, is now sufficient to establish that the
  conditions of Theorem~\ref{th:1} are satisfied as before.
\end{remark}

We now give a generalisation of Corollary~\ref{cor:1} in which the
sequence $\{A_n\}_{n\in\Z_+}$ is defined by $A_n=H_n\cap F_n$ for the
sequence $\{F_n\}_{n\in\Z_+}$ as above and with $\{H_n\}_{n\in\Z_+}$
some sequence such that each $H_n\in\si_n$.  Here additionally a
monotonicity condition is required on the sequence
$\{H_n\}_{n\in\Z_+}$.  However the monotonicity condition on the
sequence $\{F_n\}_{n\in\Z_+}$ is only required to hold in relation to
those times~$n$ at which the event $H_n$ occurs.  The result, which
reduces to Corollary~\ref{cor:1} in the case where $H_n=\Omega$ for
all $n$, is entirely natural for many applications.  It will be
applied in Example~\ref{ex:rw2} below, to the three-state contact
process of Section~\ref{sec:three-state-contact}, and to the
continuous-space ``infinite-bin'' model of
Section~\ref{sec:cont-space-model}.

\begin{corollary}\label{cor:2}
  Suppose that, for all $n$, we have $A_n=H_n\cap F_n$, where the
  sequence $\{F_n\}_{n\in\Z_+}$ satisfies the earlier conditions (F1)
  and (F2) and the sequence $\{H_n\}_{n\in\Z_+}$ is such that
  $H_n\in\si_n$ for all $n$.  Suppose further that these sequences
  satisfy the following monotonicity conditions.
  \begin{compactenum}[(a)]
  \item For all $n\ge 0$, $m>0$,
    \begin{equation}
      \label{eq:1}
      A_n \cap A_{n+m} = H_n \cap E'_{n,n+m} \cap A_{n+m},
    \end{equation}
    where the array of events $\{E'_{n,n+m}\}_{n\ge0,m>0}$ is such
    that each $E'_{n,n+m}\in\si_{n+1,n+m}$ and, for each fixed $m$,
    the sequence $\{E'_{n,n+m}\}_{n\ge0}$ is stationary.
  \item For all $n\ge 0$, $m>0$,
    \begin{equation}
      \label{eq:43}
      A_n \cap H_{n+m} = A_n \cap E''_{n,n+m},
    \end{equation}
    where the array of events $\{E''_{n,n+m}\}_{n\ge0,m>0}$ is again
    such that each $E''_{n,n+m}\in\si_{n+1,n+m}$ and, for each fixed
    $m$, the sequence $\{E''_{n,n+m}\}_{n\ge0}$ is stationary.
  \end{compactenum}
  Then the conclusions of Theorem~\ref{th:1} follow (again with
  \eqref{eq:37} holding for appropriately defined $E_{0,n}$).
\end{corollary}

\begin{proof}
  The proof is similar to that of Corollary~\ref{cor:1}, if a little
  messier.  It is necessary to verify the conditions~\eqref{eq:35} and
  \eqref{eq:36} of Theorem~\ref{th:1}, in which the sequence
  $\{F_n\}_{n\in\Z_+}$ of that theorem is as given here.  Note first
  that \eqref{eq:1} implies that also, for all $n\ge 0$, $m>0$,
  \begin{equation}
    \label{eq:2}
    H_n \cap F_n^c \cap A_{n+m} = H_n \cap (E'_{n,n+m})^c \cap A_{n+m},
  \end{equation}
  while \eqref{eq:43} similarly implies that also, for all $n\ge 0$,
  $m>0$,
  \begin{equation}
    \label{eq:3}
    A_n \cap H_{n+m}^c = A_n \cap (E''_{n,n+m})^c.
  \end{equation}
  Consider first the verification of the condition~\eqref{eq:36}.  For
  $0\le n'<n$, the event
  \begin{equation}\label{eq:5}
    A_{n'}\cap A_{n'+1}^c \cap\dots\cap A_{n-1}^c\cap A_n
  \end{equation}
  may be written as a union of events of the form
  \begin{equation}\label{eq:4}
    A_{n'}\cap B_{n'+1} \cap\dots\cap B_{n-1}\cap A_n
  \end{equation}
  where, for each $n'<k<n$, the event $B_k$ is either $H_k^c$ or
  $H_k\cap F_k^c$.  Now using equation~\eqref{eq:3} to simplify
  $A_{n'}\cap H_k^c$, equations~\eqref{eq:2} and then \eqref{eq:3} to simplify
  $H_k\cap F_k^c\cap A_n$, and finally equation~\eqref{eq:1} to
  simplify $A_{n'}\cap A_n$, it follows that each of the events given
  by \eqref{eq:4}, and so also the event given by \eqref{eq:5}, has a
  representation as $H_{n'}\cap E_{n',n}\cap F_n$, where, by
  construction, the array $\{E_{n',n}\}_{n'\ge0,n>n'}$ satisfies the
  conditions of Theorem~\ref{th:1}.

  The verification of the condition~\eqref{eq:35} of
  Theorem~\ref{th:1} (for some readily calculable sequence
  $\{H'_n\}_{n\in\Z_+}$ with each $H'_n\in\si_n$) is similar, but
  simpler.
\end{proof}

We now consider a process $\{X_n\}_{n\ge0}$ which, as indicated at the
beginning of this section, is adapted with respect to the filtration
$\{{\si}_n\}_{n\in \Z}$ (or $\{{\si}_n\}_{n\in \Z_+}$).  We further
assume that this process is defined via the specification of $X_0$ and
the stochastic recursion~\eqref{SRS}.  This is a standard situation
(but not the only one) in which the regenerative structure of the
successive blocks of the process $\{\xi_n\}$, as identified in
Theorem~\ref{th:1}, may be inherited by a (functional of the) process
$\{X_n\}$.  We introduce here a number of typical scenarios which will
be complemented by the examples of the following sections.

In what follows we wish to consider, for any time $n$, the dynamics of
the sequence $\{X_{n+i}\}_{i\ge0}$, and of functionals of this
sequence, relative to ``an initial'' $X_n$.  (One may think of a
growth model in which we are adding points at each time, or of a model
in which we center the system around the value of $X_n$ at time $n$.)
In order to do this we introduce functions $R_i(X_n,X_{n+i})$,
$i\ge1$, which capture this relative behaviour.  We assume further
that each such function $R_i$ acts as $R_i: {\cal X}^2 \to {\cal Y}$
where $({\cal X, B_X})$ is the space in which $X_n$ take values and
$({\cal Y, B_Y})$ is another measurable space.  In various of the
remaining examples of this paper, we indicate precisely the form of
the functions $R_i$.

Recall that a sequence, say $\{Y_n\}$, is \emph{stationary one-dependent}
if it is stationary and, for any $n$, the families of random variables
$\{Y_k,\, k<n\}$ and $\{Y_k,\, k>n\}$ are independent.  The following
result is a immediate extension of Theorem~\ref{th:1} and the
given conditions~\eqref{sce21}, \eqref{eq:re}, \eqref{eq:mm}.

\begin{theorem}\label{TH4}
  Suppose again that the the sequence $\{\xi_n\}_{n\in\Z}$ consists of
  i.i.d.\ random variables and that the random times
  $0\le\tau_0<\tau_1<\dots$ are defined as in Theorem~\ref{th:1},
  with all the conditions of that theorem holding.
  \begin{compactenum}[(a)]
  \item Suppose that the functions $R_i$, $i\ge1$, introduced above are
    such that, for any $n$, given that the event $A_n$ occurs, each
    random variable $R_i(X_{n+i}, X_n)$ is a measurable function of
    $\xi_{n+1},\dots,\xi_{n+i}$ only, i.e.\ for every $i\ge 1$,
    \begin{equation}\label{sce21}
      R_i(X_{n+i}, X_n)\I_{A_n}
      = g_i(\xi_{n+1},\ldots,\xi_{n+i})\I_{A_n}.
    \end{equation}
    Then the random elements
    \begin{equation}\label{eq:re}
      (R_i(X_{\tau_j+i},X_{\tau_j}),\ i=1,\dots,\tau_{j+1}-\tau_j)
    \end{equation}
    are i.i.d.\ in $j$.
  \item Suppose, more generally, that, for some fixed $m\ge 1$, the
    functions $R_i$, $i\ge1$, introduced above are now as follows: for
    any $n$, given that the event $A_n$ occurs, each random variable
    $R_i(X_{n+i}, X_n)$ is a measurable function of
    $\xi_{n-m+1},\dots,\xi_{n+i}$ only, i.e.\ for every $i\ge 1$,
    \begin{equation}\label{sce22}
      R_i(X_{n+i}, X_n)\I_{A_n}
      =g_i (\xi_{n-m+1},\ldots,\xi_{n+i})\I_{A_n}
    \end{equation}
    and that
    \begin{equation}\label{eq:mm}
      \tau_{n+1}-\tau_n \ge m \quad \text{a.s. for all $n$}.
    \end{equation}
    Then the random elements \eqref{eq:re} are stationary
    one-dependent.
  \end{compactenum}
\end{theorem}

Recall that a random sequence $\{Z_n\}$ is \emph{regenerative} if
there exist (random) times $0\le\tau_0<\tau_1<\dots$ such that the
random elements
$$
Y_0 := (\tau_0; Z_0,\ldots,Z_{\tau_0}), \
Y_1 := (\tau_1-\tau_0; Z_{\tau_0+1},\ldots, Z_{\tau_1}), \
Y_2 := (\tau_2-\tau_1; Z_{\tau_1+1},\ldots, Z_{\tau_2}), \ldots
$$
are mutually independent and the elements $\{Y_k\}_{k\ge 1}$ are
identically distributed.  The random sequence $\{Z_n\}$ is
\emph{wide-sense regenerative} if, for each  $n\ge 0$, the distribution
of the sequence $\{Z_{\tau_n+k}, k\ge 0\}$ does not depend on $\tau_n$.
Further, $\{Z_n\}$ \emph{possesses one-dependent regenerative cycles}
induced by $\{\tau_n\}$ if the sequence $\{Y_k\}_{k\ge 0}$ is one-dependent,  
and the random variables $\{Y_k\}_{k\ge 1}$  are identically distributed.
The latter two properties imply that the sequence $\{Y_k\}_{k\ge 1}$
is stationary 
and that the time instants $\{\tau_k\}$ form a renewal process.


\begin{corollary} \label{cor1}
  Suppose that the the sequence $\{\xi_n\}_{n\in\Z}$ consists of
  i.i.d.\ random variables and that the random times
  $0\le\tau_0<\tau_1<\dots$ are defined as in Theorem~\ref{th:1},
  with all the conditions
  of that theorem
  holding.  Suppose that the functions $R_i$, $i\ge1$, introduced
  above are such that either the conditions of part (a) or those of
  part (b) of Theorem~\ref{TH4} hold.  Finally suppose that the
  (common) distribution of the random variables $\tau_n-\tau_{n-1},
  n\ge 1$, is aperiodic, 
  i.e.\ $\mathrm{GCD}\{j:\Pr(\tau_1-\tau_0=j)>0 \}=1$.

  For any $n>\tau_0$, let
  \begin{equation}\label{eq:z}
    Z_n= R_{n-\tau_j} (X_n,X_{\tau_j})
    \quad \mbox{if}
    \quad \tau_j< n \le \tau_{j+1}.
  \end{equation}
  Then sequence $Z_n$ converges in the total variation norm to a
  proper limiting random variable.
\end{corollary}
Indeed, by Theorem \ref{th:1}, the sequence $Z_n$ is wide-sense
regenerative and, by Theorem \ref{TH4}, it possesses one-dependent
regenerative cycles, so that Corollary~\ref{cor1} follows from the
stability theorem for wide-sense regenerative processes, see e.g.\
Section 10 in \cite{Tho}, or \cite{Asm}.

\begin{example}\label{ex:rw2}
  \emph{Random walk with positive drift.}  We extend
  Example~\ref{ex:1} to consider a general random walk with positive
  drift.  Such a process provides possibly the simplest instance of
  the application of the above theory, and our aims here are both to
  demonstrate the use of Theorem~\ref{th:1} and to illustrate the use
  of conditioning simultaneously on both past and future events.  The
  results we give for this example are not new, but rather illustrate
  the immediate applicability of the present theory.  The sequence
  $\{\xi_n\}_{n\ge1}$ consists as usual of i.i.d.\ non-degenerate
  random variables $\xi_n$ with positive mean $a=\E\xi>0$. To make the
  example nontrivial, assume $\Pr(\xi<0)>0$.  The process of interest
  is the random walk $\{S_n\}_{n\ge0}$ with $S_0=0$ and
  $S_n=\sum_{i=1}^n\xi_i$ (we here use the notation $S_n$ instead of
  $X_n$).

  We define the earlier sequence $\{F_n\}_{n\in\Z_+}$, satisfying the
  conditions (F1) and (F2), by taking each $F_{n}$ to be the event that
  $\sum_{k=1}^m\xi_{n+k}>0$ for all $m\ge1$, i.e.\ that $S_m>S_n$ for
  all $m>n$.

  First, for a trivial application, we define the events $A_n$ of
  Theorem~\ref{th:1} by, for each $n$, $A_n=F_n$, so that the random
  times $\tau_k$, $k\ge0$, are simply the successive times of
  occurrence of the events $F_n$, and are simply the last exit times
  of the process $\{S_n\}$ above successive levels $k$.  As in the
  variant~(a) of the earlier Example~\ref{ex:1}, the
  condition~\eqref{eq:41} of Corollary~\ref{cor:1} is easily seen to
  be satisfied (with $E'_{n,n+m}=\{\sum_{k=1}^{m'}\xi_{n+k}\ge0\text{
    for all $1\le m'\le m$}\}$).  Thus the conclusions of
  Theorem~\ref{th:1} follow, and we have the well-known and elementary result
  that the segments of the process $\{S_n\}$ between the successive
  last exit times above are independent and identically distributed.

  We observe also that the functionals $R_i$ defined above are
  typically given by $R_i(S_{n+i},S_n)=S_{n+i}-S_n$.

  A more interesting application is given by defining the events $A_n$
  of Theorem~\ref{th:1} by, for each~$n$, $A_n=H_n\cap F_n$, where
  $H_n$ is the event that $S_{n'}<S_{n}$ for all $0\le n'<n$.  Thus
  the times $\tau_k$ are the successive times $n$ at which both
  $S_{n'}<S_n$ for all $n'<n$ and $S_{n'}>S_n$ for all $n'>n$.  The
  sequences of events $\{F_n\}_{n\in\Z_+}$ and $\{H_n\}_{n\in\Z_+}$
  satisfy the conditions~\eqref{eq:1} and \eqref{eq:43} of
  Corollary~\ref{cor:2}, with $E'_{n,n+m}$ as above and
  $E''_{n,n+m}=\{S_{n'}<S_{n+m}$ text{ for all $n\le n'<n+m$}\}.
  Hence again the conclusions of Theorem~\ref{th:1} follow, and we
  again have the result that the segments of the process $\{S_n\}$
  between the successive times $\tau_k$ are independent identically
  distributed.  A weak consequence is that these times themselves form
  a delayed renewal process.
  
  Here again the functionals $R_i$ of
  Section~\ref{sec:cond-infin-future} are typically given by
  $R_i(S_{n+i},S_n)=S_{n+i}-S_n$.
\end{example}

\section{The asymptotic behaviour of supercritical contact processes.}
\label{sec:right-endp-supercr}

The theory of Section~\ref{sec:cond-infin-future} has applications to
a variety of particle systems and oriented percolation models.  In
this section we consider some fairly general discrete-time contact, or
oriented percolation, processes on the integers~$\Z$.  These models
are normally studied in a continuous-time setting, and it is clear
that the relevant theory of Section~\ref{sec:cond-infin-future} could,
at the cost of some work, be adapted to that setting.

In Section~\ref{sec:two-state-contact} we study a modest
generalisation of the traditional two-state contact process on $\Z$,
in which, at each time $n$, each \emph{site} $a\in\Z$ is either
\emph{healthy} (state~$0$) or \emph{infected} (state~$1$).  We use our
earlier theory to study the behaviour of the \emph{right-endpoint
  process}, defined for each time~$n$ to be the rightmost infected
site at time~$n$.  The underlying ideas here are those of Kuczek
\cite{Kuc} and of Mountford and Sweet \cite{Mount}---the somewhat
greater generality of the model considered here makes little
difference.  However, we show that the results of Kuczek are an almost
immediate application of the theory of
Section~\ref{sec:cond-infin-future}, and become clearer when thus
understood.  (These ideas are further used by Mountford and Sweet, but
the main work of their paper is an additional ``block'' construction to
show that a certain event has a strictly positive probability---see
the discussion at the end of this section.)  This basic theory of
Section~\ref{sec:two-state-contact} is further a necessary preliminary
for applications to other particle system models.  We give one such in
Section~\ref{sec:three-state-contact} in which we study an extension
to a three-state process.  In this model, which has been considered by
a number of authors, healthy sites differ in their susceptibility to
subsequent infection according to whether they have previously been
infected.  Tzioufas \cite{Tzi} deduces right-endpoint behaviour for
the ``reverse-immunisation'' version of the process, in which
previously infected sites are easier to reinfect.  His argument uses a
monotonicity property which fails to hold in the ``immunisation''
version of the process, in which previously infected sites are more
difficult to reinfect.  We show how this difficulty is overcome by a
suitable definition of the sets $F_n$ of
Section~\ref{sec:cond-infin-future}.

\subsection{The two-state contact processes}
\label{sec:two-state-contact}

Consider a process in which sites, indexed by the integers $\Z$, are
at each time~$n$ either healthy or infected.  We define the state
$X_n$ of the process at time~$n$ to be the set of infected sites at
that time.  Between times~$n$ and $n+1$ each site $x\in X_n$ which is
infected at time $n$ produces a set $\eta_{n+1,x}\subseteq\Z$ of
\emph{descendants}, which is again a subset of $\Z$; at time $n+1$
these descendants infect a set of sites $x+\eta_{n+1,x}\subseteq\Z$,
where for any $x\in\Z$ and any $A\subseteq\Z$, we define $x+A=\{x+a:
a\in A\}$.  The state~$X_{n+1}$ of the process at time $n+1$ is given
by
\begin{equation}
  \label{eq:7}
  X_{n+1} = \bigcup_{x\in X_n} (x+\eta_{n+1,x}),
\end{equation}
i.e.\ is the union over $x\in X_n$ of the sites infected by the
descendants of these $x$.  Finally we assume that the random sets
$\eta_{n+1,x}$ are independent and identically distributed over both
times~$n$ and sites~$x$.  This model is a fairly general form of the
discrete-time version of the \emph{one-dimensional contact process},
or \emph{oriented percolation}.  Note that we do not make the
restriction (common for both discrete-time oriented percolation and
continuous-time contact processes) that, for each $n$ and $x$, the
random set $\eta_{n+1,x}$ is such that the events
$\{a\in\eta_{n+1,x}\}$ are independent over $a$.

Let $p$ be the probability that the process started with $X_0=\{0\}$,
say, \emph{survives}, i.e.\ $X_n\ne\emptyset$ for all $n\ge0$.
Suppose that $p>0$, so that the process is described as
\emph{supercritical}.  Our interest is then in the long-run behaviour
of this process.  In particular we are concerned with the behaviour of
the right-endpoint process $\{r_n\}_{n\ge0}$, conditional on survival,
where we define $r_n=\max(x:x\in X_n)$ (with $r_n=-\infty$ in the case
where $X_n$ is empty).  This, coupled with the behaviour of the
corresponding left-endpoint process, determines the growth rate of the
process.  We assume that $X_0$ is such that $r_0<\infty$ (and hence
$r_n<\infty$ for all $n$); usually $r_0=0$.

Consider first the case in which the process possesses the
following \emph{skip-free} property:  for all $n$ and all $x,y\in\Z$,
\begin{equation}
  \label{eq:13}
  x < y \quad \implies \quad x + a \le y + b
  \quad\text{for all $a\in\eta_{n+1,x}$,\ $b\in\eta_{n+1,y}$ \ a.s.}
\end{equation}
This is the discrete-time version of the \emph{nearest-neighbour}
property of the contact process on $\Z$.  The process
$\{r_n\}_{n\ge0}$ is studied by Galves and Presutti~\cite{GalPre}
and by Kuczek \cite{Kuc}.  The argument here is essentially a
rephrasing, in the framework of the present paper, of that of Kuczek,
and is given not only as an example and for completeness, but because
it is required for subsequent developments, in particular for the
theory of Section~\ref{sec:three-state-contact}, in which the present
arguments are extended and generalised.

We take the driving sequence $\{\xi_n\}_{n\ge1}$ (the index range
$n\ge1$ is sufficient), introduced in
Section~\ref{sec:cond-infin-future}, to be defined as follows.  For
each $n$, $\xi_n=\{\xi_{n,z}\}_{z\le0}$, where for each nonpositive
integer $z$, $\xi_{n,z}$ has the same distribution as any of the
random sets $\eta_{n,x}$ identified above.  In addition to being
identically distributed, the random sets $\xi_{n,z}$ are taken to be
independent over all $n$ and all $z$.  We now define the process
$\{X_n\}_{n\ge0}$ via a stochastic recursion
\begin{equation}
  \label{eq:9}
    X_{n+1}=f(X_{n},\xi_{n+1});
\end{equation}
for each $x\in X_n$ the random set $\eta_{n+1,x}$ of its descendants
at time $n+1$ is given by
\begin{equation}
  \label{eq:8}
  \eta_{n+1,x} = \xi_{n+1,x-r_n}.
\end{equation}
Thus in particular---and this is critical for the understanding of the
argument below, in which everything is in effect viewed from the right
endpoints of the processes of interest---the random set $\xi_{n+1,0}$
determines the set of descendants of the rightmost infected site $r_n$
at time~$n$, and for every other infected site at time~$n$ we count
its distance from $r_n$ in order to determine which of the random sets
$\xi_{n+1,z}$ to use for its set of descendants.  Different initial
sets $X_0$ of infected sites lead to different instances of the
process $\{X_n\}_{n\ge0}$.

For each $n\ge0$, consider the process $\{X^{(n)}_{n'}\}_{n'\ge n}$
defined by $X^{(n)}_{n}=\{0\}$ and
$X^{(n)}_{n'+1}=f(X^{(n)}_{n'},\xi_{n'+1})$ for $n'\ge n$.  Define also
the associated right-endpoint process $\{r^{(n)}_{n'}\}_{n'\ge n}$ by
$r^{(n)}_{n'}=\max(x:x\in X^{(n)}_{n'})$ (again with
$r^{(n)}_{n'}=-\infty$ in the case where $X^{(n)}_{n'}$ is empty).

We define the sequence of events $\{F_n\}_{n\in\Z_+}$ of
Section~\ref{sec:cond-infin-future} saying that the event~$F_n$ occurs
if and only if the process $\{X^{(n)}_{n'}\}_{n'\ge n}$ survives for
all future time.  It follows from the definition of the process
$\{X^{(n)}_{n'}\}_{n'\ge n}$ that the sequence $\{F_n\}_{n\in\Z_+}$
satisfies the conditions (F1) and (F2) of
Section~\ref{sec:cond-infin-future}.  In particular, the common value
of the probability of the events $F_n$ is $p$, which, by our earlier
assumption of supercriticality is strictly positive.

We define the events $A_n$ of Theorem~\ref{th:1} by, for each $n$,
$A_n=F_n$, so that the random times $\tau_k$, $k\ge0$, are simply the
successive times of occurrence of the events $F_n$.  It is our
intention to apply Corollary~\ref{cor:1}.
Fix therefore
$n\ge0$ and $m>0$ and suppose that the event $F_{m+n}$ occurs.  Then
if also the process $\{X^{(n)}_{n'}\}_{n'\ge{n}}$ survives to
time~$n+m$, i.e.\ $X^{(n)}_{n+m}\ne\emptyset$, it follows from
\eqref{eq:9} and \eqref{eq:8} that
\begin{equation}
\label{eq:14}
r^{(n)}_{n+m+n'}=r^{(n)}_{n+m}+r^{(n+m)}_{n+m+n'}
\qquad\text{for all $n'\ge0$.}
\end{equation}
We now have that, given that the event $F_{m+n}$ occurs, the event
$F_n$ occurs if and only if there occurs the event $E'_{n,n+m}$ that
the process $\{X^{(n)}_{n'}\}_{n'\ge n}$ survives to time~$n+m$.  We
thus have that the monotonicity condition~\eqref{eq:41} of
Corollary~\ref{cor:1} is satisfied (with $F_n$, $F_{n+m}$ and
$E'_{n,n+m}$ as defined here) and that, by construction, the array
$\{E'_{n,n+m}\}_{n\ge0,m>0}$ satisfies the conditions of that
corollary, so that the conclusions of Theorem~\ref{th:1} follow.
Since also the events $F_n$ have strictly positive probability $p$,
and since their indicator random variables $\I_{F_n}$ form a
stationary ergodic sequence, the first statement of
Proposition~\ref{prop:kuczek} below is now immediate from
Theorem~\ref{th:1} combined with the strong law of large numbers for
such sequences.

The proof of the second statement of Proposition~\ref{prop:kuczek} is
also essentially due to Kuczek, but, as we require essentially the
same argument (with a little extra complication) in
Section~\ref{sec:three-state-contact} below, we summarise it here.
Define the random time $\tau'=\{\min n\ge1:F_n\text{ occurs}\}$.  The
common distribution of the intervals $\tau_{k+1}-\tau_{k}$, $k\ge0$,
is that of the random time $\tau_1$ conditioned on the event
$\{\tau_0=0\}$, i.e.\ on the event $F_0$, which has strictly positive
probability.  The latter distribution is also that of the random
time~$\tau'$ conditioned on the event~$F_0$.  Hence, for the second
statement of Proposition~\ref{prop:kuczek}, it is sufficient to show
that the (unconditional) distribution of $\tau'$ is geometrically
bounded.  We show that this follows from the well-known property of
supercritical contact processes that if
$\rho=\min\{n\ge1:X^{(0)}_n=\emptyset\}$ then there exists $\alpha>0$
such that
\begin{equation}
  \label{eq:19}
  \Pr(n \le \rho < \infty) \le e^{-\alpha n}, \qquad n\ge1.
\end{equation}
To see that $\tau'$ is geometrically bounded, we may proceed forward
in time, starting at time $1$, checking at that and at selected
subsequent times $n$ whether the event $F_n$ occurs: if, at any such
time, it fails to do so we wait until the process
$\{X^{(n)}_{n'}\}_{n'\ge n}$ dies before resuming checking at
subsequent times, thereby ensuring that checks are independently
successful, each with probability $p>0$; the time to the occurrence of
a first success, and hence to the occurrence of \emph{some}
event~$F_n$, is thus a geometric sum of i.i.d.\ geometrically bounded
random variables, and is hence itself geometrically bounded, implying
the same result for $\tau'$.
We thus have the following proposition.

\begin{proposition} 
  \label{prop:kuczek}
  The successive (segments of) processes
  $\{\xi_{\tau_k+1},\dots,\xi_{\tau_{k+1}}\}$, $k\ge0$, are
  independent and identically distributed, each with finite mean
  length $1/p$.  Further the distribution of each of these lengths is
  light-tailed, i.e.\ geometrically bounded, and in particular
  possesses moments of all orders.
\end{proposition}

Now let $\{X_n\}_{n\ge0}$ be any version of the contact process
defined by \eqref{eq:9} such that $X_0\ne\emptyset$ and, if
$\{r_n\}_{n\ge0}$ is its right-endpoint process, then $r_0<\infty$.
Let $F$ be the event that the process $\{X_n\}_{n\ge0}$ survives.
Then, as in the argument above used to establish~\eqref{eq:41}, the
event $F$ occurs if and only if the process survives to time $\tau_0$.
Note also that $F_0\subseteq F$; in the extreme case where there is a
single infected site at time~$0$ we have $F_0=F$, while in the case
the number of infected sites at time $0$ is infinite, we have
$\Pr(F)=1$.  Further, from the construction of the processes involved
and recalling \eqref{eq:14}, conditional on the event $F$ and for all
$k\ge0$,
\begin{equation}\label{eq:16}
  r_{\tau_k+n'} = r_{\tau_0} +
  \sum_{j=0}^{k-1}r^{(\tau_j)}_{\tau_{j+1}}
  + r^{(\tau_k)}_{\tau_k+n'} \qquad\text{for all $n'\ge0$.}
\end{equation}
We thus have immediately the following corollary to
Proposition~\ref{prop:kuczek}.

\begin{corollary}[Kuczek]
  \label{cor:Kuczek}
  On the set $F$ the successive (segments of) processes
  $\{r_{\tau_k+1},\dots,r_{\tau_{k+1}}\}$, $k\ge0$, are
  independent and identically distributed.  Further, for some
  constant $\mu$,
  \begin{equation}
    \label{eq:10}
    \frac{r_n}{n}  \to \mu
    \quad \text{a.s., \quad as $n\to\infty$}
  \end{equation}
  and, in the Skorohod topology,
  \begin{equation}
    \label{eq:15}
    \frac{r_{[nt]}-nt\mu}{\sqrt{n}} \to B(t)
    \quad \text{in distribution, \quad as $n\to\infty$},
  \end{equation}
  where, for any $a>0$, we denote by $[a]$ the integer part of $a$,
  and where $B$ is Brownian motion with some nontrivial diffusion
  constant.
\end{corollary}

\begin{remark}
  We believe that it is worth also discussing briefly the more general
  case, considered by Mountford and Sweet~\cite{Mount}, in which the
  skip-free condition~\eqref{eq:13} is replaced by the more general
  condition that the sets $\eta_{n+1,x}$ have bounded support, if only
  as an illustration of our general thesis that everything depends of
  the appropriate definition of the sequence $\{F_n\}_{n\in\Z_+}$ of
  Section~\ref{sec:cond-infin-future}.  Here it is sufficient to
  redefine the events~$F_n$ and to show that their common probability
  remains strictly positive.  Thus, for each $n$, define not only the
  process $\{X^{(n)}_{n'}\}_{n'\ge n}$ as above (in which
  $X^{(n)}_n=\{0\}$) and its associated right-endpoint process
  $\{r^{(n)}_{n'}\}_{n'\ge n}$, but also the process
  $\{\overline{X}^{(n)}_{n'}\}_{n'\ge n}$ given by
  $\overline{X}^{(n)}_n=\Z_-$ (where $\Z_-$ is the set of nonpositive
  integers) and
  $\overline{X}^{(n)}_{n'+1}=f(\overline{X}^{(n)}_{n},\xi_{n+1})$ for
  $n'\ge n$; denote also the latter process's associated
  right-endpoint process by $\{\overline r^{(n)}_{n'}\}_{n'\ge n}$.
  Note that the process $\{\overline{X}^{(n)}_{n'}\}_{n'\ge n}$
  survives almost surely, and that, since
  $X^{(n)}_n\subset\overline{X}^{(n)}_n$, we have
  $X^{(n)}_{n'}\subseteq\overline X^{(n)}_{n'}$ for all $n'\ge n$.
  The event~$F_n$ is now defined to occur if and only if
  $r^{(n)}_{n'}=\overline r^{(n)}_{n'}$ for all $n'\ge n$.  (The
  latter condition implies the survival of the process
  $\{X^{(n)}_{n'}\}_{n'\ge n}$ and is equivalent to it in the earlier
  skip-free case.)  Further, the sequence $\{F_n\}_{n\in\Z_+}$
  continues to satisfy the conditions~(F1) and (F2) of
  Section~\ref{sec:cond-infin-future}, provided only that we can show
  that the common probability $p'$ of the events~$F_n$ is strictly
  positive.

  As before, we define the events $A_n$ of Theorem~\ref{th:1} by
  $A_n=F_n$ for all $n$, so that the times $\tau_k$ are once more the
  times of successive occurrences of the events~$F_n$.  It follows
  from the definition of the latter events that the
  condition~\eqref{eq:14} continues to hold for those $n\ge0$, $m>0$,
  such that event~$F_{n+m}$ occurs.  Thus the conditions of
  Corollary~\ref{cor:1} hold as in the skip-free case, and indeed the
  entire argument of that case holds also in the present more general
  case, subject only to the above proviso that $\Pr(F_n)>0$.  Thus, in
  this case, we once more obtain Proposition~\ref{prop:kuczek} (with
  $p$ replaced by $p'$) and Corollary~\ref{cor:Kuczek} describing the
  behaviour of the right-endpoint process.

  That $p'>0$ is shown by Mountford and Sweet~\cite{Mount}---in the
  most difficult part of their paper---using a block construction and
  under a condition on the random sets $\eta_{n+1,x}$ which limits the
  extent of the dependence between the events $\{a\in\eta_{n+1,x}\}$.
  While it seems likely that $p'>0$ in the present slightly more
  general model and that this should not be too difficult to show, we
  do not pursue this here.
\end{remark}

\subsection{A three-state contact process with immunisation}
\label{sec:three-state-contact}

We consider a model in which the susceptibility of sites to infection
depends on whether they have been previously infected.  As noted above
such models (in continuous time) have been considered by a number of
authors (Durrett and Schinazi \cite{Dur}, Stacey \cite{Sta}, Tzioufas
\cite{Tzi}).  We show here how right-endpoint, and hence growth,
behaviour can be deduced for a model with immunisation, in which
previously infected sites are more difficult to infect than those
which have not previously been infected.  What is interesting here is
that we do not have monotonicity of the process in the initial level
of infection, in that the introduction of additional infected sites at
time~$0$ may possibly, by premature infection and then immunisation of
neighbouring sites, reduce the number of infected sites at subsequent
times (see Stacey~\cite{Sta} for details).  This is in contrast to the
two-state contact process and to the ``reverse-immunisation''
three-state process mentioned earlier.  However, in the present model
there do still exist sufficient monotonicity-preserving couplings,
between instances of the process with suitably different initial
states, as to enable progress to be made with a little extra care,
notably in the definition below of the ``future'' events $F_n$ of
Section~\ref{sec:cond-infin-future} and below.  A further complication
is that the events $A_n$ of Section~\ref{sec:cond-infin-future} are no
longer simply defined by $A_n=F_n$, but rather the occurrence of the
event $A_n$ depends both on the past and the future behaviour of the
process $\{\xi_{n'}\}_{n'\ge1}$ relative to the time $n$.

We take our argument in stages: we consider first the model, then its
formulation as a stochastic recursion suitable for the application of
the theory of Section~\ref{sec:cond-infin-future}, and then the
definition of the events~$F_n$ and the times $\tau_k$ of
Section~\ref{sec:cond-infin-future}; finally we apply the earlier
theory and such additional arguments as are necessary to obtain our
results.

\paragraph{The model.}
\label{sec:model}

The varying susceptibility of sites forces a more careful
identification between sites at one time period and another.  We
therefore focus on the following generalisation of a simple oriented
percolation model, which possesses the skip-free
property~\eqref{eq:13} identified in the previous section and which is
the discrete time analogue of the three-state nearest-neighbour
contact process with a similar immunisation property.  The model
corresponds to oriented percolation through time on the integers in
which, given the state of the process at time $n$, each site which
would potentially be infected at time $n+1$ is only actually infected
with some fixed probability $q$, independently of all else, unless it
has never previously been infected, in which case it is infected with
probability one.

The state $X_n$ of the process $\{X_n\}_{n\ge0}$ at time $n$ is given by
$X_n=\{X_n(x),\,x\in\Z\}$ where each $X_n(x)\in\{-1,0,1\}$, and where
this has the interpretation:
\begin{displaymath}
  X_n(x) =
  \begin{cases}
    -1 & \quad\text{if the site $x$ is uninfected for all $n'\le n$,}\\
    0 & \quad\text{if the site $x$ is uninfected at time $n$ but has
      previously been infected,}\\
    1 & \quad\text{if the site $x$ is infected at time $n$}.
  \end{cases}
\end{displaymath}

In order to obtain a spatially symmetric model and to maintain the
above skip-free property, we make the restriction that, at each
time~$n$, the set of sites~$x$ which possibly may be infected
($X_n(x)=1$) is the set of integers $x\in\Z$ such that $n+x$ is even.
This will follow from the specification of the dynamics of the process
below provided we require that only evenly numbered sites may be
infected at time~$0$.  (In these dynamics, which we make precise
below, a site which is infected at time $n$ reverts to state $0$ at
time $n+1$; a site which is in state $-1$ or $0$ at time $n$ remains
in the same state at time~$n+1$ unless it becomes infected at time
$n+1$, in which case its state becomes $1$.)  The above restriction on
the locations of the infected sites at any time is of a purely
technical nature and is not necessary is the (more natural)
continuous-time version of the process.

Thus we assume that $X_0$ is such that $X_0(x)\ne1$ for $x$ odd.  The
state $X_{n+1}$ of the process is obtained from $X_n$ as follows: for
each $n$ and for each $x\in\Z$ such that $n+x$ is even, associate a
random set $\eta_{n+1,x}\subseteq\{-1,\,1\}$; the random sets
$\eta_{n+1,x}$ are assumed independent and identically distributed
over all $n$ and all~$x$.  Define also, for each $n$, the set of
``potentially infected'' sites $Y_{n+1}$ at time $n+1$, given by
\begin{displaymath}
  Y_{n+1} = \bigcup_{x:\;n+x \text{ even, }X_n(x)=1} (x + \eta_{n+1,x}).
\end{displaymath}
Then, for $x\in Y_{n+1}$ such that $X_{n}(x)=0$, we take
$X_{n+1}(x)=1$ with probability $q$ and $X_{n+1}(x)=0$ with
probability $1-q$, independently of all else; for $x\in Y_{n+1}$ such
that $X_{n}(x)=-1$, we take $X_{n+1}(x)=1$ with probability $1$.
For $x\notin Y_{n+1}$ we take $X_{n+1}(x)=-1$ if $X_{n}(x)=-1$ and
$X_{n+1}(x)=0$ otherwise.  (Note that, by induction, these dynamics do
indeed imply the property that, for all $n$ and for all $x$, we may
only have $X_n(x)=1$ when $n+x$ is even.  Note also that there is no
additional generality in allowing the above probability $1$, that a
never previously infected site in the set $Y_{n+1}$ becomes infected,
to be replaced by any other probability $q'\ge q$: in such a case we
may instead simply redefine the distribution of the random sets
$\eta_{n+1,x}$ to correspond to replacing each such set by the empty
set with probability $1-q'$, independently of all else; we then
replace $q'$ by $1$, and $q$ by $q/q'$, to re-express the model as an
instance of that already considered.)

We shall say that the process $\{X_n\}_{n\ge0}$ \emph{survives to
  time} $n$ if $X_n(x)=1$ for at least one $x$ and that it
\emph{survives} if it survives to all times $n\ge0$.  We assume that
the process $\{X_n\}_{n\ge0}$ is \emph{supercritical}, i.e.\ that, for any
$X_0$ such that $X_0(x)=1$ for at least one $x$, there is a strictly
positive probability that the process survives.  Note that, if $X'_0$
is obtained from $X_0$ by defining $X'_0(x)=\max(0,X_0(x))$ for all
$x$, and the resulting process $\{X'_n\}_{n\ge0}$ allowed to evolve as
above, then, in this coupling, the survival of the process
$\{X'_n\}_{n\ge0}$ implies that of the process $\{X_n\}_{n\ge0}$.  The
former process may be viewed as an instance of the basic two-state
contact process.  It follows in particular the supercriticality of the
present three-state process is equivalent to that of the two-state
process obtained as above.

Given the process $\{X_n\}_{n\ge0}$, for each $n$, define
$r_n=\max\{x: X_n(x)=1\}$ to be the \emph{right endpoint} of $X_n$
(with, as usual, $r_n=-\infty$ when $X_n(x)\ne1$ for all $x$).
Our interest in the behaviour of the process $\{r_n\}_{n\ge0}$, for
suitably chosen initial states $X_0$.  (As usual, this, taken together
with the corresponding behaviour of the left endpoint process,
characterises the growth of the process $\{X_n\}$.)

\paragraph{Formulation as a stochastic recursion and coupling.}
\label{sec:form-as-stoch}

We now reformulate the process $\{X_n\}_{n\ge0}$ as a stochastic
recursion~\eqref{SRS} as in Section~\ref{sec:cond-infin-future}.  The
i.i.d.\ driving sequence $\{\xi_n\}_{n\ge1}$ is given, for each
$n\ge0$, by the pair $\xi_{n+1}=(\xi'_{n+1},\,I_{n+1})$.  Here
$\xi'_{n+1}=\{\xi'_{n+1,z}\}_{z\le0,\,z\text{ even}}$ and each
$\xi'_{n+1,z}\subseteq\{-1,\,1\}$ is a random set with the common
distribution of the random sets $\eta_{n+1,x}$ above.  Further
$I_{n+1}=\{I_{n+1,z}\}_{z\le0,\,z\text{ even}}$ and each $I_{n+1,z}$
is an indicator random variable which takes the value $1$ with
probability $q$ and is otherwise $0$.  For each $n$ the random
elements $\xi'_{n+1}$ and $I_{n+1}$ are independent; further the
random sets $\xi'_{n+1,z}$ are independent over all $z$, as also are
the random variables $I_{n+1,z}$.  The process $\{X_n\}_{n\ge0}$ is
now updated as described above, via the stochastic
recursion~\eqref{SRS}, taking, for each $n$ and $x$,
\begin{equation}
  \label{eq:17}
  \eta_{n+1,x} = \xi'_{n+1,x-r_n},
\end{equation}
analogously to \eqref{eq:8}.  Further, given $n$, let $r'_{n+1}$ be
the rightmost point of the set $Y_{n+1}$ defined above; in the case
where $x'\in Y_{n+1}$ is such that $X_{n-1}(x')=0$ or $X_{n-1}(x')=1$,
we take
\begin{equation}
  \label{eq:18}
  X_{n+1}(x')=I_{n+1,x'-r'_{n+1}},
\end{equation}
while for $x'\in Y_{n+1}$ such that $X_{n-1}(x')=-1$ we already have
$X_{n+1}(x')=1$.

For each $n\ge0$, let $\{\hat{X}^{(n)}_{n'}\}_{n'\ge n}$ be any
version of the process $\{X_{n'}\}_{n'\ge n}$, started at time~$n$ and
defined through the above stochastic recursion~\eqref{SRS} using
\eqref{eq:17} and \eqref{eq:18} (with $\hat{X}^{(n)}_{n'}$ replacing
$X_{n'}$), in which $\hat{X}^{(n)}_{n}(0)=1$ and
$\hat{X}^{(n)}_{n}(x)=-1$ for all $x>0$.  Let
$\{\hat{r}^{(n)}_{n'}\}_{n'\ge n}$ and $\{\hat{l}^{(n)}_{n'}\}_{n'\ge
  n}$ be respectively the left and right endpoint processes associated
with this process (i.e.\
$\hat{r}^{(n)}_{n'}=\max\{x:\hat{X}^{(n)}_{n'}(x)=1\}$, with
$\hat{r}^{(n)}_{n'}=-\infty$ if no such $x$ exists, and
$\hat{l}^{(n)}_{n'}=\min\{x:\hat{X}^{(n)}_{n'}(x)=1\}$, with
$\hat{l}^{(n)}_{n'}=\infty$ if no such $x$ exists).  (Note that, for
each $n'\ge n$, it is only possible to have $\hat{X}^{(n)}_{n'}(x)=1$
at those sites $x$ such that $n'-n+x$ is even.  Since our various
processes will eventually be coupled starting from their right
endpoints, this is as it ought to be.)

Define also the particular version $\{X^{(n)}_{n'}\}_{n'\ge n}$ of the
above process, in which
\begin{displaymath}
  X^{(n)}_n(x) =
  \begin{cases}
    1 & \quad\text{if $x=0$,}\\
    0 & \quad\text{if $x<0$,}\\
    -1 & \quad\text{if $x>0$,}
  \end{cases}
\end{displaymath}
Let also $\{r^{(n)}_{n'}\}_{n'\ge n}$ and $\{l^{(n)}_{n'}\}_{n'\ge n}$
denote respectively the associated right- and left-endpoint processes.
Note that it follows from the skip-free property of the dynamics of
the process $\{X^{(n)}_{n'}\}_{n'\ge{n}}$ that, for all $n'$ to which
this process survives,
\begin{equation}
  \label{eq:32}
  X^{(n)}_{n'}(x) \ne -1 \text{ for all $x \le r^{(n)}_{n'}$},
  \qquad
  X^{(n)}_{n'}(x) = 0 \text{ for all $x <l^{(n)}_{n'}$}.
\end{equation}

We now require the following lemma, which is a simple generalisation
of the classical result for the nearest-neighbour two-state contact
process on $\Z$, and which gives the basic coupling on which the
application of the theory of Section~\ref{sec:cond-infin-future}
depends.

\begin{lemma}
  \label{lemma:coupling}
  For any instance of the process $\{\hat{X}^{(n)}_{n'}\}_{n'\ge n}$
  defined above and for the particular instance given by
  $\{X^{(n)}_{n'}\}_{n'\ge n}$, for each $n'$ to which the latter
  process survives,
  \begin{equation}\label{eq:31}
    \hat{X}^{(n)}_{n'}(x) = X^{(n)}_{n'}(x)
    \quad\text{ for all $x\ge l^{(n)}_{n'}$}.
  \end{equation}
  In particular we have $\hat{r}^{(n)}_{n'} = r^{(n)}_{n'}$ for all
  $n'$ to which the process $\{X^{(n)}_{n'}\}_{n'\ge n}$ survives.
\end{lemma}
\begin{proof}
  The proof is also a simple generalisation of that which is
  well-known for the classical two-state case, and is by induction on
  $n'\ge n$.  Thus suppose that \eqref{eq:31} holds for some
  particular $n'\ge n$ such that the process $\{X^{(n)}_{n'}\}_{n'\ge
    n}$ survives to (at least) time $n'+1$.  It then follows from the
  dynamics of the two processes involved, using also the above
  observations~\eqref{eq:32}, that, for every $x$ such that
  $X^{(n)}_{n'+1}(x)=1$, we also have $\hat{X}^{(n)}_{n'+1}(x)=1$.
  Further, since the random sets $\eta_{n+1,x}$ are subsets of
  $\{0,1\}$ (the skip-free property of the present model), any site
  $x\ge l^{(n)}_{n'+1}$ such that $\hat{X}^{(n)}_{n'+1}(x)=1$ is
  necessarily infected (at least) by some site $x'$ such that
  $X^{(n)}_{n'}(x')=1$ so that also $X^{(n)}_{n'+1}(x)=1$ (i.e.\ no
  additional infection can pass from the left of $l^{(n)}_{n'}$ at
  time $n'$ to the right of $l^{(n)}_{n'+1}$ at time $n'+1$).  Since
  also $X^{(n)}_{n'+1}(x)\ne-1$ and $\hat{X}^{(n)}_{n'+1}(x)\ne-1$ for
  all $x \le r^{(n)}_{n'+1}$, and
  $\hat{X}^{(n)}_{n'+1}(x)={X}^{(n)}_{n'+1}(x)=-1$ for all
  $x>r^{(n)}_{n'+1}$, we have that \eqref{eq:31} holds with $n'$
  replaced by $n'+1$.
\end{proof}

We shall also require below the particular instance
$\{\overline{X}_{n}\}_{n\ge 0}$ of the process
$\{\hat{X}^{(0)}_{n}\}_{n\ge 0}$ defined above (and started at
time~$0$) given by
\begin{equation}\label{eq:33}
  \overline{X}_0(x) =
  \begin{cases}
    1 & \quad\text{if $x\le0$,}\\
    -1 & \quad\text{if $x>0$.}
  \end{cases}
\end{equation}

This process is useful since, almost surely, it survives for all
time.  This enables us to make some necessary definitions without an
\emph{a priori} need to condition on survival.  Define also
$\{\overline{r}_{n}\}_{n\ge 0}$  to be the right-endpoint
process associated with the process $\{\overline{X}_{n}\}_{n\ge 0}$.

\paragraph{Definition of events $F_n$ and times $\tau_k$. }

We define the sequence of events $\{F_n\}_{n\in\Z_+}$ of
Section~\ref{sec:cond-infin-future}, analogously to
Section~\ref{sec:two-state-contact}, by saying that the event~$F_n$
occurs if and only if the process $\{X^{(n)}_{n'}\}_{n'\ge n}$
survives for all future time, i.e.\ $r^{(n)}_{n'}>-\infty$ for all
$n'\ge{n}$.  It again follows from the definition of the process
$\{X^{(n)}_{n'}\}_{n'\ge n}$ that the sequence $\{F_n\}_{n\in\Z_+}$
satisfies the conditions (F1) and (F2) of
Section~\ref{sec:cond-infin-future}.  (That the common value $p$, say,
of $\Pr(F)_n$ is strictly positive follows once more from our earlier
assumption of supercriticality.)

For each $n\ge0$, define the event
\begin{equation}\label{eq:29}
  H_{n} = \{\overline{X}_{n}(x) = -1
  \quad\text{for all $x>\overline{r}_{n}$}\}
\end{equation}
(where $\{\overline{X}_{n}\}_{n\ge 0}$ is the process defined
above with initial state $\overline{X}_0$ given by
\eqref{eq:33}).  Note that an equivalent definition is that
$H_{n}=\{\overline{r}_n\ge\overline{r}_{n'}\text{ for all $n'<n$}\}$,
i.e.\ that $n$ is such that at that time the right-endpoint process
$\{\overline{r}_{n}\}_{n\ge0}$ is at a record value.  Note also
that the event $H_{0}$ always occurs, and further that, for all $n$,
we have $H_{n}\in{\si}_{n}$.

We define the sequence of events $\{A_n\}_{n\ge0}$ of
Theorem~\ref{th:1} by, for each $n$, $A_n=H_n\cap F_n$.  As usual the
random times $\tau_k$ are the successive times of occurrence of the
events $A_n$.  (The more complex definition of the events $A_n$, in
comparison with that of Section~\ref{sec:two-state-contact}, is
required to make the right-endpoint couplings below work correctly.)

We can now state and prove the following analogue of
Proposition~\ref{prop:kuczek}.

\begin{theorem} 
  \label{thm:kuczek3}
  The successive (segments of) processes
  $\{\xi_{\tau_k+1},\dots,\xi_{\tau_{k+1}}\}$, $k\ge0$, are
  independent and identically distributed.  Further the distribution
  of each of these lengths is light-tailed, i.e.\ geometrically
  bounded, and in particular possesses moments of all orders.
\end{theorem}

\begin{proof}
  We show first that the sequences of events $\{F_n\}_{n\ge0}$ and
  $\{H_n\}_{n\ge0}$ defined above are such that the conditions of
  Corollary~\ref{cor:1} are satisfied.  Note first that it follows
  from Lemma~\ref{lemma:coupling} (applied at the time~$n$) that, for
  any $n$ such that the event $H_n$ occurs,
  \begin{equation}
    \label{eq:6}
    \overline{r}_{n'} = \overline{r}_{n} + r^{(n)}_{n'}
    \qquad\text{for all $n'\ge n$ to which the process
      $\{X^{(n)}_{n'}\}_{n'\ge n}$ survives}. 
  \end{equation}
  We show first the condition~(b) of Corollary~\ref{cor:2}.  Fix
  $n\ge0$ and $m>0$.  Given that the event $A_n$ occurs (which implies
  both the occurrence of the event $H_n$ and the survival for all time
  of the process $\{X^{(n)}_{n'}\}_{n'\ge n}$), it follows from
  \eqref{eq:6} that the event $H_{n+m}$ occurs if and only if there
  occurs the event $E''_{n,n+m}$ that the process
  $\{r^{(n)}_{n'}\}_{n'\ge n}$ is at a record value at time $n+m$,
  i.e.\ $r^{(n)}_{n+m}\ge r^{(n)}_{n+m'}$ for all $0\le m'<m$.  The
  condition~(b) of Corollary~\ref{cor:2} is now immediate, since the
  array $\{E''_{n,n+m}\}_{n\ge0,m>0}$ trivially possesses the
  properties required by that condition.
  For the condition~(a) of Corollary~\ref{cor:2}, again fix $n\ge0$
  and $m>0$.  Given that the event $H_n\cap A_{n+m}$ occurs, it
  follows from \eqref{eq:6} (both as stated and with $n$ replaced by
  $n+m$) that the event $A_n$ occurs if and only if there occurs the
  event $E'_{n,n+m}$ that the process $\{X^{(n)}_{n'}\}_{n'\ge n}$
  survives to time $n+m$. The condition~(a) of Corollary~\ref{cor:2}
  now follows, since again we have that the array
  $\{E'_{n,n+m}\}_{n\ge0,m>0}$ trivially possesses the properties
  required by that condition.
  Thus the first statement of the present theorem follows from
  Corollary~\ref{cor:2}.

  In order to prove the second statement, define the random time
  $\tau'=\{\min n\ge1:A_n\text{ occurs}\}$.  As in the corresponding
  argument for the second statement of Proposition~\ref{prop:kuczek},
  it is sufficient to show that the (unconditional) distribution of
  $\tau'$ is geometrically bounded.

  We show first that, if $\rho$ is the first time to which the
  process $\{X^{(0)}_{n}\}_{n\ge 0}$ fails to survive, then, as for
  the basic two-state contact process, there exists $\alpha>0$ such
  that
  \begin{equation}
    \label{eq:20}
    \Pr(n \le \rho < \infty) \le e^{-\alpha n}, \qquad n\ge0.
  \end{equation}
  For each $n\ge0$ define the process $\{\tilde{X}^{(n)}_{n'}\}_{n'\ge
    n}$ via the above stochastic recursion~\eqref{SRS} and
  \eqref{eq:17}, \eqref{eq:18}, with the initial state
  $\tilde{X}^{(n)}_{n}$ given by $\tilde{X}^{(n)}_{n}(0)=1$ and
  $\tilde{X}^{(n)}_{n}(x)=0$ for all other $x$.  Let $\tilde{\rho}$ be
  the minimum value of $n\ge0$ such that the process
  $\{\tilde{X}^{(n}_{n'}\}_{n'\ge n}$ (started at time~$n$) survives.
  Suppose now that the process $\{X^{(0)}_{n'}\}_{n'\ge 0}$ survives
  to time $\tilde{\rho}$; then, in yet another instance of the
  coupling arguments used above (in which processes are ``matched''
  from their right endpoints), it follows from the above stochastic
  recursion that the process $\{X^{(0)}_{n'}\}_{n'\ge 0}$ necessarily
  survives for all time.  We deduce that if $\rho<\infty$ then
  necessarily $\rho<\tilde{\rho}$, and thus we conclude that, for all
  $n\ge0$,
  \begin{equation}\label{eq:21}
    \Pr(n \le \rho < \infty) \le \Pr(\tilde{\rho} > n).
  \end{equation}
  However, given the driving sequence $\{\xi_n\}_{n\ge1}$, the
  successive processes $\{\tilde{X}^{(n)}_{n'}\}_{n'\ge n}$ are simply
  successive instances of the basic two-state nearest-neighbour
  contact process, in each case started with a single infective,
  launched by the sequence $\{\xi_n\}_{n\ge1}$ exactly as in
  Section~\ref{sec:two-state-contact}.  We have already observed in
  that section that the time required to initiate such a process which
  survives (the time to the first of the events $F_n$ of that section)
  is geometrically bounded.  The required conclusion~\eqref{eq:20} now
  follows from this and \eqref{eq:21}.

  To complete the proof, we need to show that the distribution of
  $\tau'$ is geometrically bounded.  We argue as in
  Section~\ref{sec:two-state-contact}, again with a little extra
  complication.  We again proceed forward in time, starting at time
  $\nu_1=1$ and checking at that, and at selected subsequent times
  $\nu_k$, $k>1$, such that the event $H_{\nu_k}$ occurs, whether
  the event $F_{\nu_k}$ also occurs; if it fails to do so we wait
  until the process $\{X^{(\nu_k)}_{n'}\}_{n'\ge\nu_k}$ dies out---which
  we have shown it then does in a time which is geometrically
  bounded---before resuming checking for $F_{\nu_{k+1}}$ at the first
  \emph{subsequent} time~$\nu_{k+1}$ such that $H_{\nu_{k+1}}$
  occurs.  It follows, again from the right-endpoint coupling of
  Lemma~\ref{lemma:coupling} as in the first part of the present
  proof, that, for each such $k\ge1$, the increment
  $\overline{r}_{\nu_{k+1}}-\overline{r}_{\nu_k}$ is equal
  to the maximum value attained by the right endpoint of the process
  $\{X^{(\nu_k)}_{n'}\}_{n'\ge\nu_k}$ prior to its dying out, and so
  this increment is geometrically bounded.  Further, from the
  construction, the successive increments
  $\overline{r}_{\nu_{k+1}}-\overline{r}_{\nu_k}$ are
  i.i.d..  Let $K$ be the number of checks required to obtain a
  success (the event $F_{\nu_K}$ occurs).  Then, since, each of the
  above checks is independently successful with probability $p>0$, the
  random variable $K$ is geometrically distributed, independently of
  the above increments in the right-endpoint
  process~$\{\overline{r}_{n}\}$.  Thus
  $\overline{r}_{\nu_K}-\overline{r}_1$ is a geometric sum
  of i.i.d.\ geometrically bounded random variables, and so
  $\overline{r}_{\nu_K}$ is geometrically bounded.

  Now let $\{\tilde{r}_{n}\}_{n\ge0}$ be the right-endpoint of the
  three-state process $\{\tilde{X}_{n}\}_{n\ge0}$ in which
  $\tilde{X}_0(0)=1$ and $\tilde{X}_0(x)=0$ for $x\ne0$.
  Then, from the usual coupling
  $\tilde{r}_{n}\le\overline{r}_{n}$ for all $n$, and so
  also $\tilde{r}_{\nu_K}$ is geometrically bounded.  However,
  $\{\tilde{X}_{n}\}_{n\ge0}$ is simply an instance of the
  supercritical two-state contact process, in which the set of
  initially infected sites is $\Z_-$.  It follows easily from the results of
  Kuczek for the regenerative behaviour of this process (as given in
  the previous section) that, since $\tilde{r}_{\nu_K}$ is
  geometrically bounded, the random variable $\nu_K$ is
  itself geometrically bounded.  Since $\nu_K$ is the time~$n$ to the
  occurrence of \emph{some} event $A_n$, the result that $\tau'$ is
  geometrically bounded now follows.
\end{proof}

Finally, consider again any instance $\{\hat X_n\}_{n\ge0}$ of our
three-state process, defined by the stochastic recursion~\eqref{SRS}
via using \eqref{eq:17} and \eqref{eq:18} as above, in which the
initial state $\hat X_0$ is such that $\hat{X}_0(x_0)=1$ for some
$x_0$ and $\hat{X}_0(x)=-1$ for all $x>x_0$.  Again let
$\hat{r}_{n'}=\max\{x:\hat{X}_{n'}(x)=1\}$ be its associated
right-endpoint process.  Let $F$ be the event that the process
$\{\hat{X}_n\}_{n\ge0}$ survives.  It follows from the earlier
coupling for this process (with the time $\tau_0$ replacing the
time~$0$) that event $F$ occurs if and only if the process $\{\hat
X_n\}_{n\ge0}$ survives to time $\tau_0$.  Analogously to the
situation for the two-state process, we have that $F_0\subseteq F$,
and that, conditional on the event $F$ and for all $k\ge0$,
\begin{equation}\label{eq:22}
  \hat r_{\tau_k+n'} = \hat r_{\tau_0} +
  \sum_{j=0}^{k-1}r^{(\tau_j)}_{\tau_{j+1}}
  + r^{(\tau_k)}_{\tau_k+n'} \qquad\text{for all $n'\ge0$.}
\end{equation}
Thus, again as for the two-state process, we have the following
corollary to Theorem~\ref{thm:kuczek3}.

\begin{corollary}
  \label{cor:kuczek3}
  For the process $\{\hat X_n\}_{n\ge0}$ and on the set $F$ defined
  above, the successive (segments of) processes $\{\hat
  r_{\tau_k+1},\dots,\hat r_{\tau_{k+1}}\}$, $k\ge0$, are independent
  and identically distributed.  Further, for some constant $\mu$,
  \begin{equation}
    \label{eq:30}
    \frac{\hat r_n}{n}  \to \mu
    \quad \text{a.s., \quad as $n\to\infty$}
  \end{equation}
  and, in the Skorohod topology,
  \begin{equation}
    \label{eq:26}
    \frac{\hat r_{[nt]}-nt\mu}{\sqrt{n}} \to B(t)
    \quad \text{in distribution, \quad as $n\to\infty$},
  \end{equation}
  where again $B$ is Brownian motion with some nontrivial diffusion
  constant.
\end{corollary}

Similar behaviour holds for the left-endpoint process for any process
$\{X_n\}_{n\ge0}$ whose initial state $X_0$ is such that $X_0(x_0)=1$
for some $x_0$ and $X_0(x)=-1$ for all $x<x_0$.  Thus, finally, for
any process $\{X_n\}_{n\ge0}$ whose initial state $X_0$ is such that
$X_0(x)=-1$ for all $x$ outside some finite interval, we may deduce
the behaviour, conditional on its survival, of both its left and right
endpoints.  In particular, conditional on its survival, the growth
rate of the process is given by $2\mu$, where $\mu$ is as given by
Corollary~\ref{cor:kuczek3}.

\section{Infinite-bin models}
\label{sec:infinite-bin-model-1}

In this Section, we consider a discrete-space infinite-bin model 
and its continuous-space analogue. 

In the discrete setting, we review
the basic model introduced and studied in \cite{FoKonst} 
(see also \cite{DFKonst, FMS}). We recall a stability result from
\cite{FoKonst} (see Proposition
\ref{ThBin1} below), with a new proof, 
and provide a new generalisation (see Theorem
\ref{ThNew1}). We show that the  both results may be considered as applications
of the techniques developed in Section
\ref{sec:cond-infin-future}.

Then we introduce a new continuous-space model and prove a new stability 
result there (see Theorem \ref{thm:cont1}), by applying again the methodology 
from Section
\ref{sec:cond-infin-future}.

\subsection{Discrete-space infinite-bin model}
\label{sec:DIBModel}

\subsubsection{Basic model}
\label{sec:BasicM}

Consider an infinite number of bins arranged on the line and indexed,
say, by the non-positive integers. Each bin can contain an unlimited
number of particles (we assume it to be finite for the moment). A configuration
is either a finite-dimensional vector
$
x= (x_{-l},\ldots,x_0)
$
where $x_{i}$ is the number of particles in bin $i$,
or an infinite-dimensional vector
$
x= (\ldots, x_{-l}, \ldots, x_0 ).
$

The indexing by non-positive integers is convenient because
we are interested in the asymptotic behaviour of a finite number of right-most
coordinates of vectors representing a stochastic recursion.
At each integer step, precisely
one particle---the \emph{active} particle---of the current
configuration is chosen according to some
rule (to be given below). If the particle is in
bin $-i\le -1$, then a new particle is created and placed in bin $-i+1$. Otherwise, if the chosen particle
is in bin $0$ then a new bin is created to hold the ``child particle'' and
a relabelling of the bins occurs: the existing bins are shifted
by one place to the left (are are re-indexed) and the new bin is given the
label $0$.

To be more precise, define the configuration space ${\cal X}$
as the set of all infinite-dimensional vectors
$x= (\ldots, x_{-2},x_{-1},x_0)$ with non-negative integer-valued
coordinates, which have the following property: if $x_{-l}>0$,
then $x_{-l+1}>0$.  In other words, either all the coordinated
of a configuration vector are strictly positive or there is only
a finite number of non-zero coordinates, say $l+1$ -- then these are
coordinates $x_{-l},x_{-l+1}, \ldots, x_0$.
We endow ${\cal X}$ with the natural topology
of pointwise convergence, and let ${\cal B_X}$ be the corresponding class
of Borel sets generated by this topology.

The {\it extent} of an $x\in {\cal X}$ is defined as
$|x|=l$, if there is $l+1$ non-zero coordinates,
$x= (\ldots, 0,0,x_{-l},\ldots,x_0)$,
with the ${\cal L}_1$ {\it norm}
$$
||x|| = \sum_{j=0}^l x_{-j},
$$
and if all the coordinates of $x$ are positive, we set
$|x|=||x||= +\infty$.

Let ${\N}$ be the set of positive integers.
The dynamics of the model may be defined using the map
$f: {\cal X}\times {\N} \to {\cal X}$ where
\begin{eqnarray*}
f(x,\xi ) &=& [x,1] \quad \mbox{if} \quad \xi\le x_0,\\
&=& x+e_{-k} \quad \mbox{if} \quad
\sum_{j=0}^kx_{-j} < \xi \le \sum_{j=0}^{k+1} x_{-j}, \ 0\le k < |x|,\\
&=&x+e_{-|x|}, \quad \mbox{if} \quad \xi > ||x||.
\end{eqnarray*}
Here $[x,1]$ is a concatenation of the vector $x$ with $1$, i.e.\ if
$x=(\ldots, x_{-l}, \ldots, x_0)$, then
$[x,1]=(\ldots, y_{-l-1},y_{-l},\ldots,y_{-1},y_0)$
where $y_0=1$ and $y_{-j-1}=x_{-j}$, for $j\ge 0$. Further,
$e_{-j}$ is the infinite unit vector whose $(-j)$th coordinate is $1$
with all other coordinates equal to $0$. 
Then, given an
i.i.d.\
sequence $\{\xi_n\}_{n\in {\Z_+}}$ of
${\N}$-valued random variables and an ${\cal X}$-valued random variable
$X_0 = (\ldots, X_{-k,0},X_{-k+1,0},\ldots,X_{-1,0},X_{0,0})$,
we define a stochastic recursion by
\begin{equation}\label{bin1}
X_{n+1}=f(X_n,\xi_{n+1}), \quad n\ge 0
\end{equation}
where $X_n = (\ldots, X_{-k,n},X_{-k+1,n},\ldots,X_{-1,n},X_{0,n})$.

Verbally the dynamics may be explained as follows. Each time $n$ we number
again
the existing particles from the rightmost bin to the leftmost (so, if
$X_n$ takes value $x=(\ldots,0,x_{-l},\ldots, x_0)$, then the
particles in the rightmost bin are numbered $1$ to $x_0$,
in the next bin they are numbered $x_0+1$ to $x_0+x_{-1}$, and so
on). Then the random variable $\xi_n$ is the
number of the active particle defined in the earlier description.

Fix a non-negative integer $k$ and let  $X_n(-k)$ be the $(k+1)$-dimensional projection of $X_n$,
$$
X_n(-k) = (X_{-k,n},X_{-k+1,n},\ldots ,X_{0,n}).
$$
The following result may be found in \cite{FoKonst}.

\begin{proposition}\label{ThBin1}
Assume that $\{\xi_n\}_{n\in\Z_+}$ is an i.i.d.\ sequence.  Assume
also that
$\Pr(\xi_i =1)>0$ and ${\E} \xi_i < \infty$.
Then, 
for any integer $k\ge 0$,  $X_n(-k)$ converges to a proper limiting random vector in the total variation
norm. Therefore, $X_n$ weakly converges to its proper limit.
\end{proposition}

Based on the theory from Section \ref{sec:cond-infin-future}, one can provide a short
alternative proof of Proposition \ref{ThBin1}. 
We start with the simplest case $k=0$.
In this case, the proof is based on Corollaries \ref{cor:1} and
\ref{cor1}. 

In order to avoid trivialities, assume $\Pr (\xi_i =1) <1$.

Let the functions $R_i$ of Section~\ref{sec:cond-infin-future} be
given by $R_i(X_{n+i},X_n)=X_{0,n+i}$. Further, let
\begin{equation}\label{eq:Fni}
F_n =
\bigcap_{i\ge 1} \{\xi_{n+i}\le i\} \equiv \bigcap_{i\ge 1} F_{n,i}.
\end{equation}
Note that, since $\E\xi_i<\infty$, the events $F_n$ have strictly
positive probability, and indeed satisfy the conditions (F1) and (F2)
of Section~\ref{sec:cond-infin-future}.  Define now, for each $n$, the
event $A_n=F_n$.  Clearly
$$
F_n \cap F_{n+m} = \bigcap_{i=1}^m F_{n,i} \cap F_{n+m},
$$
and so the condition \eqref{eq:41} of Corollary \ref{cor:1}
holds. Further, the condition
\eqref{sce21} of Theorem \ref{TH4} holds because, given the event $F_n$, the future
process of placing particles is the same for all histories up to time $n$.

Finally, the aperiodicity condition of Corollary \ref{cor1} follows since
\begin{eqnarray*}
 \Pr (\tau_{n+1}-\tau_n=1) &=&
\sum_l \Pr (\tau_n=l,\tau_{l+1}=l+1) \\
&=& \sum_l \Pr (\tau_n=l) \Pr (F_{l+1} \ | \ F_l) = \Pr (F_1 \ | \ F_0)\\
&=& \prod_{i=2}^{\infty} \Pr (\xi_i\le i-1 \ | \ \xi_i \le i) 
\ge   \prod_{i=2}^{\infty} \Pr (\xi_i\le i-1) > 0.
\end{eqnarray*}
The proof of Proposition \ref{ThBin1} now follows in the case $k=0$
from Corollaries~\ref{cor:1} and \ref{cor1}.

For the
proof for general $k>0$, we need events of the form $B_n=H_n\cap F_n$
where the events $F_n$ are again as given by \eqref{eq:Fni}
and
$$
H_n =\bigcap_{1\le i\le k} \{\xi_{n+1-i}=1\}. 
$$
We may observe that, given $H_n$, we have $X_n(-k) = (1,1,\ldots,1)$.

Now we define the events $A_n$ as follows: 
$A_n=\emptyset$, for $n<2k$ and, for $n\ge 2k$,
$$
A_n =\bigcap_{i=1}^{k-1} B_{n-i}^c \cap B_n
$$
which may be represented as $A_n = E_{n-k,n}\cap F_n$, for a stationary
sequence $E_{n-k,n}\in\si_{n-k+1,n}$. Finally, one may take 
$R_i(X_{n+i},X_n) = X_{n+i}(-k)$. 
Then all conditions of Theorem \ref{TH4} are satisfied, and the result
again follows from Corollary~\ref{cor1}, on noting that 
once more aperiodicity follows from the condition ${\Pr}(\xi_i =1)>0$.

\begin{remark}
This model has close links to the model from \cite{Comets},
see \cite{FoKonst} for more detail.
\end{remark}

\subsubsection{Extension of the basic model}

Consider the infinite bin model introduced in the Section~\ref{sec:BasicM} and
let $p_i = {\Pr} (\xi =i)$. One of the main conditions in 
Proposition  \ref{ThBin1} is that $p_1>0$. We assume now that this condition is
violated and that instead the following condition
holds: there exist two positive
integers $1<i_1<i_2$ such that
\begin{equation}\label{twopositive}
p_{i_1}>0 \quad \mbox{and} \quad p_{i_2}>0.
\end{equation}
Then the following statement holds.

\begin{theorem}\label{ThNew1}
Assume that $\{\xi_n\}_{n\in\Z_+}$ is an i.i.d.\ sequence with a
common finite mean
${\E} \xi_i$. Assume also that $p_1=0$ and that the condition \eqref{twopositive} holds.
Assume further that the numbers $i_1$ and $i_2$ are mutually prime.  
Then, 
for any integer $k\ge 0$,  $X_n(-k)$ converges to a proper limiting random vector in the total variation
norm. Therefore, $X_n$ weakly converges to its proper limit.
\end{theorem}

The proof of Theorem~\ref{ThNew1} will be based on the following
simple observation (see, e.g., \cite{ChRa}).

\begin{lemma}\label{chra}
For any two integers $1<i_1<i_2$, there exist a positive integer $m$
and a sequence of integers $j_1,j_2,\ldots,j_{m-1} \in \{i_1,i_2\}$
such that, for any $n>m$
and for any vector $X_{n-m}$ as in Section~\ref{sec:BasicM},
we have
$$
X_{n,0} {\I}_{B_n} \ge i_1 {\I}_{B_n} \quad \mbox{a.s.}
$$
where the events $B_n$ are defined as
\begin{equation}
  \label{eq:11}
  B_n= \{\xi_{n}=i_2\} \cap \bigcap_{l=1}^{m-1} \{\xi_{n-m+l}=j_l\}
  \cap \{\xi_{n-m}=i_2\}.
\end{equation}
 \end{lemma}

\begin{proof}[Proof of Theorem~\ref{ThNew1}]

By the conditions of the theorem, the stationary events $B_n$ defined
by~\eqref{eq:11} have a positive
probability.

Let $r= j_1(k+1)$. 
For $n\le r$, we let $A_n=\emptyset$. 
For $n>r$, let $A_n=H_n\cap F_n$, with
$$
H_n = B_{n-r}\cap D_n \quad \mbox{where} \quad
D_n= \bigcap_{1\le l \le r}\{\xi_{n+1-l}=i_1\},
$$
and where  
\begin{equation}\label{Fn12}
F_n = \bigcap_{l\ge 1} \{\xi_{n+l}\le i_1 + l-1\}.
\end{equation}
Clearly, for $n>r$, given the occurrence of the event $H_n$, all the
coordinates of the vector
$X_n(-k)$ are equal to $i_1$. Thus, given the event $A_n$, the placings of 
the particles numbered $n+1,n+2,\ldots$ do not depend on the left tail of
the vector $X_n$ 
nor on the past values of the vector $X_j$, $j<n$.

One can check directly that both the conditions \eqref{eq:1}
and \eqref{eq:43} are satisfied.

We may now define the functions~$R_i$ of
Section~\ref{sec:cond-infin-future} by $R_i(X_{n+i},X_n)=X_{n+i}(-k)$.
Then the condition \eqref{sce21} holds which implies the conclusion
\eqref{eq:re} of part (a) 
of Theorem \ref{TH4}. 

Observe further that $B_{n+l}\cap D_n  =\emptyset$, for
any $1\le l \le r$. Therefore, $A_n\cap A_{n'}=\emptyset$ for all $n<n'$ with
$n'-n\le r+i_2-i_1$ and ${\mathbf P} (A_n\cap A_{n'})>0$ if
$n'-n>r+i_2-i_1$. The latter implies the aperiodicity condition of 
Corollary~\ref{cor1}, 
and the required result follows.
\end{proof}

\subsection{Continuous-space model with varying link lengths}
\label{sec:cont-space-model}

In this section, we introduce and study
a new model which is a continuous-space extension of the infinite-bin
model, and which has applications in, for example, queueing theory.
As described below, it may be viewed as a model for the locations of
points on the negative real line, in which at each successive time
precisely one
of these points gives birth to a further point, and in which it is
convenient to associate a \emph{link} between this child point and its
parent.  We thus think of it as a random links model.
Once again, our aim is to study the asymptotic behaviour of this
model as ``seen from the right''.

Before introducing the new model, we remark that the
basic model of
Section \ref{sec:BasicM} may
be described slightly differently. Namely, we may assume that, at each time $n$,
particle number $-j$ may be active with some probability, say $p(-j)$.
Each active particle proposes to put a new particle in the bin next to
its own
(in other words, at distance 1 to the right),
and the rightmost active particle wins. If particles become active
independently, then this description coincides with the description proposed
earlier if we let ${\Pr} (\xi > j) = \prod_{i=1}^j (1-p(-i))$.

Now assume, for simplicity, that all the $p(-j)$ are equal and
introduce the following  
continuous-space extension of the model, in which the positions of
particles are real-valued:
at time $n$, each active particle (say particle $-j$) proposes a location
for the new particle at a random
distance, $l_{n,-j}$ to the right of particle $-j$ (here
the $l_{n,-j}$ needs not be integer),
and the rightmost proposed location (say that proposed by particle
$-j_0$) wins. Then we say there is a link of length
$l_{n,-j_0}$ from particle $-j_0$
to the new particle.

\begin{remark} One may view this model as a model of a system with
infinitely many servers and with random constraints.
There is an infinite queue in front;
each successive client $n$ is allocated to a free server, but the
start of its service is delayed by the maximum of times $l_{n,-j}$
of all previous clients that are active, see e.g.\ \cite{FoKonst}
and the references therein for further comments.
\end{remark} 
 
\begin{remark} One can consider various natural generalisations of
  this setting where 
 the same methodology may be easily applied. For example, we 
may assume  that, at any time, the first $K\ge 0$ particles cannot be
active and that all the others
become active independently and either with equal probabilities $p\in (0,1]$
or with varying probabilities.
\end{remark}

Here is a formal description of the model.
Let ${\cal X}$ be the space of left-infinite vectors of the
form ${\bf x} = (\ldots , x_{-k},x_{-k+1},\ldots,x_{-1},x_0)$
where $x_0=0$ and $x_{-k}\le x_{-k+1}$, for all $k\ge 1$. We also
assume that either all the coordinates of $\bf{x}$ are finite or
some of them may be equal to $-\infty$. In the latter
case, due to the monotonicity, there will be only finite number of
finite coordinates, say,
$x_{-j}=-\infty$ for all $j> k$ and $x_{-j}>-\infty$ for all $j\le k$,
for some $k=0,1,\ldots$. Then we write for short
${\bf x} = (x_{-k},\ldots,x_0)$. We denote by ${\cal X}_0$ the space of
finite-dimensional vectors (which may be viewed as a subspace of ${\cal X}$).

Let ${\cal L}$ be the space of infinite sequences
${\bf l}=(\ldots,l_{-k}.l_{-k+1},\ldots,l_0)$ consisting of non-negative
real-valued elements, and let ${\cal Q}$ the space of
infinite sequences of the form
${\bf q}=(\ldots, q_{-k},q_{-k+1},\ldots,q_0)$ where each $q_{-k}\in \{0,1\}$.

Introduce the function
$$
f: {\cal X}_0 \times {\cal L}\times {\cal Q} \rightarrow {\cal X}_0
$$
using the following rule.
For ${\bf x}=(x_{-k},\ldots,x_0)$, let
$$
h:= h({\bf x,l,q})= \max_{i: q_{-i}=1} (x_{-i}+l_{-i})
$$
and
$$
h:=h({\bf x,l,q})= x_{-k}
$$
if $q_{-i}=0$, for all $0\le i \le k$.

If $h\le 0$ and, say, $x_{-j}\le h \le x_{-j+1}$, for some $j$, then
$$
f({\bf x,l,q}) =(x_{-k},\ldots,x_{-j},h,x_{-j+1},\ldots,x_0)
$$
and if $h>0$, then
$$
h({\bf x,l,q}) = (x_{-k}-h, x_{-k+1}-h, \ldots, x_0-h, 0).
$$
In other words, if $h\le 0$, we add an extra coordinate $h$, and
if $h>0$, we again add the coordinate and then subtract $h$ from all
coordinates of the new vector.

\begin{remark}
An equivalent way to describe the dynamics is to use point processes.
Instead of considering vectors, we may consider finite ordered sequences
of points, with the right-most point at $0$.
\end{remark}

Now we introduce stochastic assumptions. 
Let $\{{\bf l}_n\}$ and
$\{{\bf q}_n\}$
be two i.i.d.\ 
sequences of vectors that do not depend on each other.
Assume also that each ${\bf l}_n$ and each ${\bf q}_n$
consists of i.i.d.\ random variables, $l_{n,j}$ and $q_{n,j}$.
Let $q= {\mathbf P} (q_{n,j}=0)=1-p$ with $p={\mathbf P} (q_{n,j}=1)$.
Assume that ${\mathbf P} (l_{0,0}> 0)=1$ 
and that 
\begin{equation}\label{secondm}
{\bf E} \left(l_{0,0}\right)^2 < \infty,
\end{equation}
and let
$$
a= {\mathbf E} l_{0,0}.
$$

Recall that vectors $X_n$ always have infinitely many coordinates. 
Our model is now defined by starting  
from a fixed vector $X_0\in {\cal X}_0$  
and running the a stochastic
recursion
\begin{equation}
  \label{eq:12}
  X_{n+1}= f(X_n,{\bf l}_n,{\bf q}_{n}).
\end{equation}

Our aim is now to establish the following analogue of
Theorem~\ref{ThNew1} for the discrete-space model.

\begin{theorem}\label{thm:cont1}
For any $j\ge 0$, the finite-dimensional projections $(X_{n,-j},\ldots,
X_{n,0})$ of vectors $X_n$ converge to a proper limiting vector in the
total variation norm.
\end{theorem} 

\begin{proof}
  Let $\nu_n= \min \{i: q_{n,-i}=1\}$. Then $\{\nu_n\}$ is an i.i.d.\
  sequence with a common geometric distribution.  It is convenient to
  us to assume this sequence to be doubly-infinite, $-\infty <n <
  \infty$.

  Analogously to Section \ref{sec:BasicM}, introduce the events
$$
F_n^{(1)}= \bigcap_{j\ge 1} \{\nu_{n+j}\le j \}
$$
and conclude that these events form a stationary ergodic sequence,
each with a strictly positive probability
$$
{\bf P} \left(F_0^{(1)}\right) = \prod_{j\ge 1} (1-q^j) >0,
$$
and, moreover, satisfy the monotonicity condition \eqref{eq:41}.
Thus, by Corollary~\ref{cor:1}, the times
$0 < T_1 < T_2 < \ldots$ of occurrences of the events $F_n^{(1)}$ form
a stationary renewal sequence. We may easily extend this sequence to
stationary renewal sequence $\ldots < T_{-1} < T_0 \le 0 < T_1 < T_2 <
\ldots$ on the whole real line.

Further, the i.i.d.\ cycle lengths $t_k=T_{k+1}-T_k$, $k\ne 0$, have a
light-tailed distribution (i.e. have a finite exponential moment) and,
therefore, a finite positive mean $b = {\mathbf E} t_1$.  Also, the
cycle $T_1-T_0$ has a light-tailed distribution.


Assume for simplicity that the initial vector $X_0$ corresponds to a
single particle at the origin with all the others at $-\infty$.
Number this particle~$0$.  Each subsequent configuration $X_n$ adds
precisely one further, finitely located, particle to the configuration
$X_{n-1}$ (with the existing particles relocated if necessary).
Number this particle~$n$.  Thus particles are numbered in the order of
their creation, and are assumed to keep their numbering for all
subsequent times (including when they are relocated).
Now colour ``red'' all particles numbered $T_1,T_2,\ldots$, i.e.\
those created at the times of occurrence of the events $F_n^{(1)}$;
colour ``green'' the remaining particles.  For each $n$ and for each
$k$ such that $n>T_{k+1}$, consider the relative locations of the
particles in vector $X_n$. The following observations are clear:

\begin{compactitem}
\item{} the red particle $T_k$ is located to the left of the red
  particle $T_{k+1}$, and the distance between them is a random
  variable, say, $d_k$ which is stochastically bigger than the
  ``typical'' link $l_{0,0}$; in particular, ${\mathbf E} d_k \ge
  a>0$;
\item{} all the green particles numbered $T_{k}+1,\ldots,T_{k+1}-1$
  are located between these two red particles;
\item{} the relative locations and, in particular, the distances
  between particles numbered $T_k,\ldots,T_{k+1}$ stay the same, for
  all $n>T_{k+1}$.
\end{compactitem}

Therefore, if $n=T_{k+1}$ for some $k\ge 0$, then the $t_k$ last
coordinates of vector $X_n$ take values between $-d_k$ and $0$, next
(to the left) $t_{k-1}$ coordinates take values between
$-d_{k-1}-d_{k}$ and $-d_k$, $\ldots$, $t_k$ coordinates take values
between $-d_1-d_2-\ldots - d_k$ and $-d_2-\ldots -d_k$, and then $T_1$
coordinates are smaller than $-d_k-\ldots - d_1$ (recall that there
are also infinitely many coordinates equal $-\infty$).  Therefore, the
vector $X_n$ is smaller (coordinate-wise) than the vector, say $Y_n$,
with infinitely many finite coordinates where the last $t_k$
coordinates equal $0$, the next $t_{k-1}$ coordinates equal $-d_{k}$,
$\ldots$, $t_{1}$ coordinates equal $-d_2-\ldots - d_k$, $t_0$
coordinates equal
$-d_k-\ldots - d_1$, $t_{-1}$ coordinates equal $-d_k-\ldots - d_0$,
etc.

Further, we may define vectors $Y_n$ for all $n$ (and not only for
those with $\I_{F^{(1)}_n}=1$) as follows: if $T_k\le n < T_{k+1}$ for
some $k$, we obtain the vector $Y_n$ by concatenating the vector
$Y_{T_k}$ with $n-T_k$ coordinates equal to zero, i.e.\ $Y_n =
(\ldots, Y_{T_k,-1}, 0,0,\ldots,0)$.  Then, clearly, $Y_n\ge X_n$,
coordinate-wise, for all $n\ge 0$.

Since the sequence $\{T_n\}$ is stationary and renewal, the vectors
$\{Y_n\}$ form a stationary ergodic sequence.  For any $n$, let
$T_{n,0}\le n$ be the last time of occurrence of the events
$F^{(1)}_k$ before or at time $n$, $T_{n,-1}<T_{n,0}$ the previous
such time, and so on.  Further, let $d_{n,i}$ be the distance between
particles $T_{n,i-1}$ and $T_{n,i}$ in the vector $Y_n$.  By
stationarity, the random variable $T_{n,0}-n$ has the same
distribution as $T_0$ and the random vectors
$\{(T_{n,i}-T_{n,i-1},\,d_{n,i})\}$, $i\le 0$, do not depend on
$T_{n,0}$ and are i.i.d., with the same distribution as $(t_1,d_1)$.

Let $b_0 = {\mathbf E} |T_{0}|$. For $\varepsilon \in (0,1)$, consider
the following events
\begin{equation}\label{HH}
  H_n = \{n-T_{n,0}\le b_0(1+\varepsilon )
  \cap \bigcap_{i\le 0} \{T_{n,i}-T_{n,i-1}\le b(1+\varepsilon ),\,
  d_{n,i}\ge a(1-\varepsilon )\}.
\end{equation}
These events form a stationary ergodic sequence and, by the strong law
of large numbers, have a positive probability, for any $\varepsilon
>0$. Further, one can see that, for $n'<n$, if the events $H_{n'}$ and
$F_{n'}^{(1)}$ occur, with $T_{n,j}=n'$ for some $j\le 0$, then, for
event $H_{n}$ to occur, it is sufficient for \eqref{HH} to hold only
for $j$ between $i$ and $0$.  Namely,
\begin{equation}\label{HFTE}
  H_{n'}\cap F_{n'}^{(1)}\cap \{T_{n,j}=n'\}\cap H_n
  = 
  H_{n'}\cap F_{n'}^{(1)}\cap \{T_{n,j}=n'\}\cap E_{n',n}
\end{equation}
where the event $E_{n',n}$ belongs to the sigma-algebra
$\sigma_{n',n}$ and does not depend on $j$. Let $\widehat{A}_n=
H_n\cap F_n^{(1)}$.  Taking the union in all $j$ in \eqref{HFTE}, we
obtain condition \eqref{eq:43} with $\widehat{A}_n$ in place of
$A_n$. Condition \eqref{eq:1}, again with $\widehat{A}_n$ in place of
$A_n$, may be verified similarly.

Let constants $c_{-j}$, $j\ge 0$, be defined as follows:
$$
c_{-j} = 0 \quad \mbox{for} \quad 0\le j \le b_0(1+\varepsilon ) +
b(1+\varepsilon )
$$
and, for $r\ge 1$,
$$
c_{-j} = ra(1-\varepsilon ) \quad \mbox{for} \quad (1+\varepsilon
)(b_0+rb) <j \le (1+\varepsilon )(b_0+(r+1)b).
$$
Introduce now a second ``future'' event
\begin{equation}\label{F2}
  F_n^{(2)} =
  \{l_{n+1,0} \ge \sup_{j>0} (l_{n+1,-j}+c_{n,-j})\}
  \cap \{\forall \ \ i>1,
  l_{n+i,\nu_{n+i}} \ge
  \sup_{j>i} (l_{n+1,-j}+c_{n,-j})
\end{equation}
Clearly, for each $n$, the events $H_n$, $F_n^{(1)}$ and $F_n^{(2)}$
are mutually independent.  Further, the events $F_n^{(2)}$ form a
stationary ergodic sequence and, by \eqref{secondm}, have a strictly
positive probability.  The meaning of the event $F_n^{(2)}$ is: all
locations for ``new'' particles (with numbers greater than $n$)
generated by ``old'' particles (with numbers less than $n$)
are relatively small; thus, given the simultaneous occurrence of the
three events $H_n$, $F_n^{(1)}$ and $F^{(2)}_n$, all future links
(starting from time $n$) are established only between particles
numbered $n, n+1,n+2,\ldots$.

Let $F_n=F_n^{(1)}\cap F_n^{(2)}$ and let $A_n=H_n\cap F_n$.  We may
conclude that the stationary sequence of events $\{F_n\}$ satisfy
properties (F1) and (F2).  Further, an extra intersections with events
$F_n^{(2)}$ preserves properties \eqref{eq:43} and \eqref{eq:1}, so
the conclusion of Corollary \ref{cor:2} holds.

Thus the conclusions of Theorem \ref{th:1} holds.  Further, it is easy
to verify aperiodicity for the times $\tau_n$ defined in Theorem
\ref{th:1}, because ${\mathbf P} (\tau_2=1)>0$.  Then we may take the
random functions $R_i$ of Section~\ref{sec:cond-infin-future} to be
given by $R_i(X_{n+i},X_n)= (X_{n,-j},\ldots,X_{n,0})$, for any fixed
$j$, and conclude that the conditions of Theorem \ref{TH4} and
Corollary \ref{cor1} are satisfied too.  The result now follows from
the latter corollary.
\end{proof}

\section{Relation to Harris ergodicity} 
\label{sec:HR}

In this Section, we revisit the basic concept of Harris ergodicity
and show that it may be considered as a particular
case of the approach developed in Section~\ref{sec:cond-infin-future}.

It is known (see, e.g., \cite{Kif, BoFo, Bo98})
that a time-homogeneous Markov chain $\{X_n\}$ taking values in a
measurable state space
$({\cal X, B_X})$ may be represented as a stochastic
recursion~\eqref{SRS}
with i.i.d.\  driving sequence $\{\xi_n\}$ if one assumes ${\cal B_X}$
to be countably generated. Moreover, without loss of generality, one
may assume that the random variables 
$\xi_n$ are real-valued and uniformly distributed on the interval $(0,1)$.

Recall 
the following classical definition.

{\bf Definition.}  A time-homogeneous {\it aperiodic} Markov chain
$\{X_n\}$ taking values in a general state space $({\cal X,B_X})$ is
\emph{Harris ergodic} (or \emph{Harris positive recurrent}) if the
following conditions
hold:\\
there exist a set $V\in {\cal B_X}$, a number $0<p\le 1$, an integer
$m\ge 1$, and a probability
measure $\varphi$ on $({\cal X,B_X})$ such that
\begin{compactenum}
\item[(1a)] if $\tau\equiv\tau(V):=\min\{n\ge 1\ :\ X_n\in V\}$ is the
  first hitting time of the set $V$, then,
for any $x\in {\cal X}$,
\begin{displaymath}
  \Pr_x(\tau<\infty) = 1.
\end{displaymath}
\item[(1b)] $\sup_{x\in V} \E_x\tau < \infty$,
\end{compactenum}
and
\begin{compactenum}
\item[(2)] for any $x\in V$,
$$
\Pr_x(X_m\in \cdot ) \ge p\varphi (\cdot ),
$$
\end{compactenum}
where $\Pr_x$ and $\E_x$ denote respectively probability and
expectation conditional on $\{X_0=x\}$.
Note that frequently the set $V$ is called {\it positive recurrent} if
the conditions (1a)--(1b) hold.

The following result holds (see e.g. \cite{MyTw}).

\begin{proposition}\label{Th1}
  Assume that the Markov chain $\{X_n\}$ is Harris ergodic.  Then
  there exists a unique stationary (invariant) distribution
  $\pi$, which is also limiting in the sense of convergence in the
  total variation norm: for any $X_0=x\in {\cal X}$,
\begin{equation}\label{TV}
\sup_{B\in{\cal B_X}} |\Pr_x (X_n\in B)-\pi (B)|\to 0, \quad \mbox{as}
\quad n\to\infty.
\end{equation}
Conversely, if \eqref{TV} holds for any initial value $X_0=x\in {\cal X}$,
then the Markov chain is Harris ergodic.
\end{proposition}

The ``coupling-type'' interpretation of the dynamics of a Harris ergodic
Markov chain was proposed in \cite{AN, Num}, see also \cite{BoFo, Bo98}. 
This may be done as follows:
we run a Markov chain until it hits set the $V$ (say at time $n$), then we
flip a coin (independently of everything else) with probability~$p$ of
getting a head.
If this happens, then we say that $n+m$ is
the success time when the Markov chain ``forgets its past'', i.e.\ $X_{n+m}$ has
distribution $\varphi$ 
which is independent
of what has happened before time $n$ (but may depend on what has
happened at times $n+1,...,n+m-1$). If, on the contrary, we get a tail  (which
occurs with probability $1-p$), we wait
for the first time
after time $n+m$ when the Markov chain visits $V$ again and then flip
independently another
coin.
After a geometric number of trials, we come to a success with probability one.
It is well-known (see, e.g. \cite{Asm} or \cite{MyTw}) that the Harris ergodic
Markov chain may be made {\it regenerative} if $m=1$ and {\it wide-sense
regenerative} and possessing the {\it one-dependence} property if $m\ge 2$ 
(the definitions are given in 
Section \ref{sec:cond-infin-future} after Theorem 4). More precisely, let
$0=T_0 < T_1 < T_2 < \ldots$ be the times of successes. Then
the \emph{cycle lengths} $T_{i+1}-T_i$ are i.i.d.\ in $i\ge 0$, and the
\emph{cycles}
$\left(T_{i+1}-T_{i},\, X_{T_i},X_{T_i+1},\ldots,X_{T_{i+1}-1}\right)$
are i.i.d.\ in $i\ge 1$ if
$m=1$, and are 1-dependent 
and identically distributed
(for $i\ge 1$) if $m\ge 2$. The one-dependence follows since if
$m\ge 2$, then the set $\{n+1,\ldots,n+m-1\}$ is non-empty, and if,
say,
$T_i=n+m$ for some $i$, then the values $X_{n+1},\ldots,X_{n+m-1}$
belong to the $i$th cycle, but they also depend, in general, on the value
$X_{n+m}$ that belongs to the $(i+1)$st cycle.

For self-containedness, we recall in more detail the
coupling construction of Athreya and Ney \cite{AN}, in the particular
case $m=1$, see \cite{Tho} for the general case.
Let $P(x,B)$ be
the transition kernel of the Markov chain. 
Then, using the condition~(2) in the above definition of Harris
ergodicity, for $x\in V$,
$$
P(x,B) = p \varphi (B) + (1-p) \frac{P(x,B)-p\varphi (B)}{1-p}
\equiv p\varphi (B) + (1-p) Q(x,B)
$$
where $Q$ is also a transition probability kernel. 

Now we provide the coupling construction.
First, let $\{ \alpha_n\}$ be an i.i.d.\ sequence with common distribution
${\mathbf P} (\alpha_n=1)=1-{\mathbf P} (\alpha_n=0)=p$.
Second, let $\{\zeta_n\}$ be another i.i.d.\ sequence (having, say, a
distribution which is uniform on
$(0,1)$) that does not depend on
$\{\alpha_n\}$.  Further, let $g_1:{\cal X}\times (0,1) \to{\cal X}$
and $g_2: V\times (0,1)\to {\cal X}$
be two measurable functions such that $g_1(x,\zeta_n)$ has distribution
$P(x,\cdot )$ and $g_2(x,\zeta_n)$ has distribution $Q(x,\cdot )$.
Finally, let $\{\psi_n\}$ be a third independent i.i.d.\ sequence with
distribution $\varphi$.

Then the dynamics of $X_n$ is defined as follows:
\begin{compactenum}[]
\item if $X_n\in V$ and $\alpha_{n+1}=1$, then $X_{n+1}=\psi_{n+1}$;
\item if $X_n\in V$ and  $\alpha_{n+1}=0$, then $X_{n+1}=g_2(x,\zeta_{n+1})$;
\item if $X_n\in \overline{V}$, then $X_{n+1}=g_1(x,\zeta_{n+1})$.
\end{compactenum}
Clearly, $X_n$ may be represented as a stochastic recursion with an i.i.d.\
driving sequence $\xi_n = (\zeta_n,\alpha_n,\psi_n)$.

Therefore, for $m=1$, Proposition \ref{Th1} may be viewed as a particular case of Corollary \ref{cor1}, with
$H_n=\{X_n\in V\}$, $F_n = \{\alpha_{n+1}=1\}$, $A_n=H_n\cap F_n$,
$\tau_n=T_n$ and
$R_i(X_{\tau_n+i},X_{\tau_n})=X_{\tau_n+i}$.  This follows since 
condition \eqref{eq:41} and then the conditions~\eqref{eq:1} and
\eqref{eq:43} are immediately verified.

In the case $m>1$, we need a slightly more elaborated coupling construction
to conclude that again Proposition \ref{Th1} may be viewed as a particular case of Corollary \ref{cor1}.

\section{Comments}\label{sec:discussion}

There is an extensive list of other examples, and we mention here
a few only.

First, there are directions where the methodology may be
applied directly:
Markov chains with long memory (e.g.\ \cite{Comets,Gal,DeSaPi});
exited random walks (e.g.\ \cite{BW,BR,Men4});
modified random walks (e.g.\ \cite{Kesten}).

Second, there are models which involve conditioning on the infinite future
which---in contrast with examples considered in this paper---has
probability zero in the original probability space, e.g.\
\cite{BB}. 

A further interesting example of embedded regenerative
structure is of shifts of Brownian motions, see 
\cite{Last}.

In the case where the future event $F_n$ admits a representation
$$
F_n = \bigcap_{k\ge n} F_{n,k}
$$
where $F_{n,k} \in \sigma_{n,k}$ satisfy the monotonicity property
\eqref{eq:41}, 
one can introduce a general scheme for ``perfect simulation''
of the process along the lines of, say, \cite{FoKonst}.

We also comment that our results may be directly
extended onto more general models where either the elements
of the
driving sequence $\{\xi_n\}$  remain independent but are only ``asymptotically
identically distributed'' or where this sequence is
regenerative or, more generally, converges (in an appropriate manner) to a
regenerative sequence. Here the renovation method (see e.g. \cite{Bo98}),
or the method of renovating events, may be of use.

\section*{Acknowledgements}
\label{sec:acknowledgements}

The authors are most grateful to the referees for their insightful and
helpful comments.

\end{document}